\documentclass[11pt]{amsart}

\usepackage[utf8]{inputenc}
\usepackage{microtype}
\usepackage{amsfonts,comment}
\usepackage{amsmath}
\usepackage{graphicx}
\usepackage{float}
\usepackage[dvipsnames]{xcolor}
\usepackage{hyperref}
\usepackage{tikz-cd}

\usepackage{amssymb}
\usepackage{enumitem}
\usepackage{mathrsfs}

\usepackage{tikz}
\usetikzlibrary{topaths}
\usetikzlibrary{calc}

\usepackage{a4wide}

\usepackage{empheq}

\newtheorem{theorem}{Theorem}[section]
\newtheorem*{theorem*}{Theorem}
\newtheorem{lemma}[theorem]{Lemma}
\newtheorem{proposition}[theorem]{Proposition}
\newtheorem*{proposition*}{Proposition}
\newtheorem{corollary}[theorem]{Corollary}
\newtheorem*{corollary*}{Corollary}

\newtheorem*{conjecture*}{Conjecture}
\newtheorem{question}[theorem]{Question}
\newtheorem*{question*}{Question}

\theoremstyle{definition}
\newtheorem{definition}[theorem]{Definition}
\newtheorem*{definition*}{Definition}
\newtheorem{remark}[theorem]{Remark}

\newtheorem{observation}[theorem]{Observation}

\newcommand{\N}{\mathbb{N}}
\newcommand{\Z}{\mathbb{Z}}
\newcommand{\Q}{\mathbb{Q}}
\newcommand{\R}{\mathbb{R}}

\newcommand{\euler}{\mathcal{X}}

\newcommand{\calH}{\mathcal{H}}
\newcommand{\calF}{\mathcal{F}}
\newcommand{\calP}{\mathcal{P}}
\newcommand{\calU}{\mathcal{U}}
\newcommand{\PSA}{\mathrm{PSAut}}
\newcommand{\PSO}{\mathrm{PSOut}}
\newcommand{\SA}{\mathrm{SAut}}
\newcommand{\SO}{\mathrm{SOut}}
\newcommand{\PB}{\mathrm{PB}}
\newcommand{\sphere}{\mathbb{S}}

\newcommand{\Scomb}{$\Sigma^2_{\rm{comb}}$}
\newcommand{\Scombp}{$(\Sigma^2_{\rm{comb}})$}

\DeclareMathOperator{\Hom}{Hom}
\DeclareMathOperator{\Aut}{Aut}
\DeclareMathOperator{\Out}{Out}
\DeclareMathOperator{\Inn}{Inn}

\DeclareMathOperator{\MMO}{MMO}
\DeclareMathOperator{\MMA}{MMA}
\DeclareMathOperator{\WO}{WO}
\DeclareMathOperator{\WA}{WA}
\DeclareMathOperator{\ZO}{ZO}
\DeclareMathOperator{\ZA}{ZA}

\DeclareMathOperator{\F}{F}
\DeclareMathOperator{\FP}{FP}
\DeclareMathOperator{\supp}{supp}

\DeclareMathOperator{\Cay}{Cay}

\DeclareMathOperator{\Cos}{Cos}

\numberwithin{equation}{section}

\begin{document}

\title{Dense and empty BNSR-invariants of the McCool groups}
\date{\today}
\subjclass[2020]{Primary 20F65;   
                 Secondary 57M07} 

\keywords{BNSR-invariant, automorphisms of free groups, McCool group, loop braid group, finiteness properties}

\author[M.~Ershov]{Mikhail Ershov}
\address{Department of Mathematics, University of Virginia, Charlottesville, VA}
\email{ershov@virginia.edu}

\author[M.~C.~B.~Zaremsky]{Matthew C.~B.~Zaremsky}
\address{Department of Mathematics and Statistics, University at Albany (SUNY), Albany, NY}
\email{mzaremsky@albany.edu}

\begin{abstract}
An automorphism of the free group $F_n$ is called \emph{pure symmetric} if it sends each generator to a conjugate of itself. The group $\PSA_n$ of all pure symmetric automorphisms and its quotient $\PSO_n$ by the group of inner automorphisms are called the \emph{McCool groups}. In this paper we prove that every BNSR-invariant $\Sigma^m$ of a McCool group is either dense or empty in the character sphere, and we characterize precisely when each situation occurs. Our techniques involve understanding higher generation properties of abelian subgroups of McCool groups, coming from the McCullough--Miller space. We also investigate further properties of the second invariant $\Sigma^2$ for McCool groups using a general criterion due to Meinert for a character to lie in $\Sigma^2$.
\end{abstract}

\maketitle
\thispagestyle{empty}

\section{Introduction}

Let $G$ be a finitely generated group. The \emph{BNS-invariant} $\Sigma(G)$ 
is a fundamental geometric object introduced in the seminal 1987 paper of Bieri, Neumann, and Strebel \cite{bieri87}. In 1988, two families of its higher-dimensional generalizations were defined: \emph{homological BNSR-invariants} $\Sigma^m(G,R)$, $m\in\N$, where $R$ is a commutative unital ring, with $\Sigma^1(G,\mathbb Z)=\Sigma(G)$, studied by Bieri and Renz in \cite{bieri88}, and \emph{homotopical BNSR-invariants} $\Sigma^m(G)$, $m\in\N$, with $\Sigma^1(G)=\Sigma(G)$, studied by Renz in \cite{renz_thesis}. In this paper we will only deal with homotopical invariants and refer to them simply as \emph{BNSR-invariants}. The $m$th invariant $\Sigma^m(G)$ is defined whenever $G$ is of \emph{type $F_m$}, meaning $G$ admits a classifying space with finite $m$-skeleton. In general the BNSR-invariants of $G$ form a descending chain $\Sigma^1(G)\supseteq \Sigma^2(G)\supseteq \cdots$ of open subsets of the sphere $\sphere(G) = S^d$, where $d = \dim \Hom(G,\R)-1$. Among other things, $\Sigma^m(G)$ completely determines which coabelian subgroups of $G$ are of type $\F_m$ (we call a subgroup of $G$ \emph{coabelian} if it contains the commutator subgroup $[G,G]$); in particular, $[G,G]$ itself is of type $\F_m$ if and only if $\Sigma^m(G)=\sphere(G)$.

While the BNS-invariant $\Sigma(G)$ has been fully computed in a number of interesting cases, the class of groups for which all BNSR-invariants $\Sigma^m(G)$ are defined and completely determined for all $m$ is much smaller. One prominent class for which the latter has been accomplished is right-angled Artin groups (RAAGs) \cite{meier98}, and it is natural to ask what can be said about BNSR-invariants for groups which are sufficiently close to RAAGs in some sense. Examples of such groups where some progress has been made include Artin groups in general \cite{meier01,blascogarcia22,escartin-ferrer24} and Bestvina--Brady groups \cite{kochloukova22}. In this paper we will focus on 
the computation of BNSR-invariants for yet another collection of groups similar to RAAGs which are frequently called \emph{McCool groups}.

Let $F_n$ be the free group of rank $n$, with a fixed basis $\{x_1,\dots,x_n\}$. An automorphism of $F_n$ is called \emph{pure symmetric} if it sends each $x_i$ to a conjugate of itself. The subgroup $\PSA_n$ of the automorphism group $\Aut(F_n)$ consisting of all pure symmetric automorphisms arises in many contexts and has several different names. It is sometimes simply called the \emph{pure symmetric automorphism group}, sometimes the \emph{McCool group}, in reference to seminal work of McCool \cite{mccool86}, and sometimes the \emph{pure loop braid group}, for its appearance as the group of motions of $n$ unlinked distinguishable circles in $3$-space; see for example \cite{baez07,brendle13} for some important applications of these groups and \cite{damiani17} for an overview of all their different guises. The ``outer'' version $\PSO_n$ of $\PSA_n$ is the quotient of $\PSA_n$ by the group of inner automorphisms, and we will collectively refer to all the $\PSA_n$ and $\PSO_n$ as \emph{McCool groups}.

At a very superficial level, McCool groups are similar to RAAGs because they admit simple presentations by generators and relations where all relations are of the form $[u,v]=1$ for some words $u$ and $v$, and in the majority of relations $u$ and $v$ are both generators. However, there are much deeper analogies between these groups, which make many techniques and ideas from \cite{meier98} applicable to McCool groups.

The first BNS-invariant for the McCool groups has been completely determined by Orlandi-Korner in \cite{orlandikorner00}. In \cite{zaremsky18}, the second author computed some pieces of the higher BNSR-invariants $\Sigma^m(\PSA_n)$ and conjectured that each $\Sigma^m(\PSA_n)$ is either dense (in the respective sphere) or empty, depending on whether $m\leq n-2$ or not \cite[Remark~4.24]{zaremsky18}. The main theorem of this paper confirms this conjecture and also establishes the analogous result for the groups $\PSO_n$:
\begin{theorem}\label{thm:mainPSA}
For all $n\ge 2$, the following hold:
\begin{itemize}
\item[(a)] $\Sigma^{n-2}(\PSA_n)$ is dense in $\sphere(\PSA_n)$. 
\item[(b)] $\Sigma^{n-1}(\PSA_n)=\emptyset$.
\end{itemize}
For all $n\ge 3$, the following hold:
\begin{itemize}
\item[(c)] $\Sigma^{n-3}(\PSO_n)$ is dense in $\sphere(\PSO_n)$.
\item[(d)] $\Sigma^{n-2}(\PSO_n)=\emptyset$.
\end{itemize}
Hence, $\Sigma^k(\PSA_n)$ is dense for all $0\leq k\le n-2$ and empty for all $k\ge n-1$, and $\Sigma^k(\PSO_n)$ is dense for all $0\leq k\le n-3$ and empty for all $k\ge n-2$.
\end{theorem}

\begin{remark}$\empty$
\begin{itemize}
\item[(i)] Since we always have $\sphere(G)=\Sigma^0(G)\supseteq \Sigma^1(G)\supseteq\cdots$, the ``hence'' part is immediate from parts (a)--(d).
\item[(ii)] All assertions of Theorem~\ref{thm:mainPSA} were previously known for $n\leq 3$. Indeed, $\Sigma^{1}(\PSA_2)$ is empty
and $\Sigma^{1}(\PSA_3)$ is dense in $\sphere(\PSA_3)$ by the main theorem of \cite{orlandikorner00}, and the proof of that theorem also
easily implies that $\Sigma^{1}(\PSO_3)$ is empty. Finally, the equality $\Sigma^{2}(\PSA_3)=\emptyset$ was established by the second author
in \cite{zaremsky18}.
\end{itemize}
\end{remark}

While Theorem~\ref{thm:mainPSA} naturally extends the results and
confirms the aforementioned conjecture from \cite{zaremsky18}, the techniques used in the proofs are very different, as we will see later.

A consequence of Theorem~\ref{thm:mainPSA}, together with Theorem~\ref{thrm:bnsr_fin_props} below, is that if $\chi\colon \PSA_n\to \Z$ is a ``generic'' discrete character of $\PSA_n$, then the kernel of $\chi$ is of type $\F_{n-2}$ but not $\F_{n-1}$. Similarly, the kernel of a ``generic'' discrete character of $\PSO_n$ is of type $\F_{n-3}$ but not $\F_{n-2}$. We should also mention that our proof of parts (b) and (d) actually shows the stronger statement that $\Sigma^{n-1}(\PSA_n;\Z)=\emptyset$ and $\Sigma^{n-2}(\PSO_n;\Z)=\emptyset$, where these are the homological BNSR-invariants.

\medskip

In order to prove parts (b) and (d), about certain BNSR-invariants being empty, we prove a rather easy (and somewhat ``folklore''), but very general, result about groups of type $\F$ (meaning groups admitting a finite classifying space) with non-zero Euler characteristic -- see Proposition~\ref{prop:euler} which asserts that if $G$ is a group of type $\F$ whose BNSR-invariants are symmetric (Definition~\ref{def:symmetric}) and whose Euler characteristic is non-zero, then $\Sigma^\infty(G)=\emptyset$. Since for a group of type $\F$ we have $\Sigma^\infty(G)=\Sigma^m(G)$ where $m$ is the dimension of a finite classifying space for $G$, parts (b) and (d) follow easily from known facts about the McCool groups.

Proving parts (a) and (c), about certain BNSR-invariants being dense, requires quite a lot more effort. The key is a criterion due to Meier, Meinert and Van Wyk \cite{meier98} that relates the BNSR-invariants of a group to the BNSR-invariants of its subgroups, using the notion of higher generation of groups by families of subgroups introduced by Abels and Holz in \cite{abels93}; see Theorem~\ref{thrm:MMV98}. The Meier--Meinert--Van Wyk criterion involves a number of hypotheses on the group, and in order to apply it to the McCool groups we make use of the Whitehead posets and McCullough--Miller spaces (see Definitions~\ref{def:whitehead} and~\ref{def:mm}) and the fact that the groups are highly generated by abelian subgroups; see Proposition~\ref{prop:contractible}. In a recent paper \cite{ardaizgale}, a generalization of McCullough--Miller space for RAAGs is constructed with all these desired properties, so it would be interesting to use this to analyze the BNSR-invariants of the analogs of the McCool groups for RAAGs.

\medskip

Given that $\PSA_n$ is sometimes called the ``pure loop braid group'', it is natural to ask if one could prove similar results about the pure braid groups $\PB_n$. Since $\PB_n$ splits as $(\PB_n/Z(\PB_n))\times \Z$ \cite[Section~9.3]{farbmargalit}, $\Sigma^\infty(\PB_n)$ is dense in its character sphere for trivial reasons, but an interesting problem is understanding when $\Sigma^m(\PB_n/Z(\PB_n))$ is dense or empty. In fact, $\PB_n/Z(\PB_n)$ is isomorphic to the pure mapping class group of the $(n+1)$-punctured sphere (equivalently the $n$-punctured plane) \cite[Section~9.3]{farbmargalit}, and it is already known that $\Sigma^{n-2}(\PB_n/Z(\PB_n))=\emptyset$ for all $n\ge 3$ \cite[Example~7.3]{hughes_kielak} (when $n<3$ the group is trivial). Thus, a natural conjecture is that $\Sigma^{n-3}(\PB_n/Z(\PB_n))$ is dense in its character sphere. This is trivially true for $n=3$ and can be confirmed for $n=4$ thanks to the computation of $\Sigma^1(\PB_4)$ in \cite{koban15}. An impediment to proving this conjecture for larger $n$ is a lack of a known analog for $\PB_n$ of the McCullough--Miller space for $\PSA_n$ and $\PSO_n$ and, more generally, a lack of a known highly generating family of abelian subgroups. If one could find a family of abelian subgroups of $\PB_n$ that is $\infty$-generating in the sense of Definition~\ref{def:hi_gen}, then proving  that $\Sigma^{n-3}(\PB_n/Z(\PB_n))$ is dense would probably be within reach. We should mention that some results about highly generating families for $\PB_n$ are in the appendix of \cite{bux16}, but the degree of high generation is not large enough to tell us anything about $\Sigma^{n-3}(\PB_n/Z(\PB_n))$.

\medskip

We also prove further results about the second BNSR-invariant $\Sigma^2$ for McCool groups. In dimension 2, an underutilized result from the Ph.D. thesis of Meinert \cite{Meinertthesis} leads to an easy-to-check sufficient condition for a character class to lie in $\Sigma^2(Q)$, where $Q$ is a quotient of a well-understood group $G$ -- see Theorem~\ref{thm:Meinertcentral} for the general criterion and Proposition~\ref{prop:Sigma2sufficient} for the specific application to McCool groups. Since Meinert's thesis is not easily accessible, we also provide a proof of Meinert's theorem; see Theorem~\ref{thm:Meinert} in Appendix~\ref{app:meinert}.

\medskip

This paper is organized as follows. In Section~\ref{sec:prelims}, we recall some background material on BNSR-invariants and McCool groups and set the stage for inspecting BNSR-invariants of McCool groups. In Section~\ref{sec:empty}, we prove parts (b) and (d) of Theorem~\ref{thm:mainPSA}, about BNSR-invariants being empty. In Section~\ref{sec:dense}, we prove parts (a) and (c), about BNSR-invariants being dense. In Section~\ref{sec:Sigma2}, we prove our further results about $\Sigma^2$, stemming from Meinert's criterion, and then in Appendix~\ref{app:meinert}, we provide a proof of Meinert's criterion.

\subsection*{Acknowledgments} The authors are grateful to Sam Hughes, Dawid Kielak, and Lorenzo Ruffoni for helpful discussions and pointers to references. We also thank the referee for many helpful suggestions. MZ was supported for part of the time spent on this project by Simons grant \#635763.

\section{Preliminaries}\label{sec:prelims}

In this section we recall some background material on BNSR-invariants in general and on the groups $\PSA_n$ and $\PSO_n$ that will be our main objects of interest.

\subsection{BNSR-invariants}\label{ssec:invariants}

In this subsection, we recall the definition and some important properties of the Bieri--Neumann--Strebel--Renz (BNSR) invariants of a group.
The invariants are somewhat difficult to define and usually quite difficult to compute, but reveal a wealth of information about the group. Let us gradually build up their definition.

Recall that a \emph{classifying space} for a group $G$ is a connected CW-complex whose fundamental group is isomorphic to $G$ and whose homotopy groups vanish in all dimensions higher than $1$. Equivalently, it is a quotient of a connected CW-complex whose homotopy groups in positive dimensions are trivial by a free cellular action of $G$.

A group is said to be of \emph{type $F_n$} if it admits a classifying space with finite  $n$-skeleton. Every group is of type $\F_0$, finite generation is equivalent to type $\F_1$, and finite presentability is equivalent to type $\F_2$. A group is of \emph{type $\F_\infty$} if it is of type $\F_n$ for all $n$, and if a group has a finite classifying space, it is said to be of \emph{type $F$}.

Throughout the paper, by a \emph{character} of a group $G$ we will mean a homomorphism $\chi\colon G\to\R$ from $G$ to the additive group of real numbers. Two characters are \emph{equivalent} if they differ by multiplication by a positive scalar. The equivalence classes $[\chi]$ of non-trivial characters of $G$ form the \emph{character sphere} of $G$, denoted $\sphere(G)$.

This sphere should be viewed as the boundary of the Euclidean vector space $\Hom(G,\R)$. If this space is $d$-dimensional, then the character sphere of $G$ is $(d-1)$-dimensional. BNSR-invariants of $G$ are certain subsets of $\sphere(G)$, which can be defined in several slightly different ways. The definition below is close, but not identical, to the original definition from \cite{renz_thesis}.
\vskip .15cm

Let $G$ be a group of type $\F_n$, so by definition it has a classifying space with finite $n$-skeleton. Choose any such classifying space $K$, let $\widetilde{K}$ be the universal cover of $K$, and let $X$ be the $n$-skeleton of $\widetilde{K}$, with the canonical $G$-action. It is not hard to show that for any non-trivial character $\chi$ of $G$ there exists a map $h_\chi\colon X \to \R$ (called a \emph{height function}) such that $h_\chi(g.x)=\chi(g) + h_\chi(x)$ for all $g\in G$ and $x\in X$. For $t\in\R$, define $X^{h_\chi\ge t}$ to be the subcomplex of $X$ consisting of all cells whose vertices lie in $h_\chi^{-1}([t,\infty))$.

\begin{definition}[BNSR-invariants]
\label{def:BNSR-invariants}
Let $n$, $G$ and $X$ be as above. For any $0\leq m\leq n$, the $m$th \emph{Bieri--Neumann--Strebel--Renz (BNSR) invariant} $\Sigma^m(G)$ is the subset of $\sphere(G)$ consisting of all $[\chi]$ for which the filtration $X^{h_\chi\ge t}$ of $X$ is \emph{essentially $(m-1)$-connected}, meaning that for every $t$ there exists $u\le t$ such that the inclusion $X^{h_\chi\ge t}\to X^{h_\chi \ge u}$ induces the trivial map in all homotopy groups up to dimension $m-1$.
\end{definition}

Note that by construction, $X$ in Definition~\ref{def:BNSR-invariants} is $(n-1)$-connected, so all $k$-spheres in $X$ can be filled for any $k\le n-1$, but what the condition $[\chi]\in\Sigma^m(G)$ measures is whether they can be filled without causing the $\chi$ value to drop too far in the negative direction somewhere along the filling disk. This definition is independent of the choices of $X$ and $h_\chi$; see for example \cite[Theorem~12.1]{bieri03}.
\vskip .12cm

\begin{remark}
\label{rem:BNSR-invariants}~
\begin{itemize}
\item[(a)] One can compute BNSR-invariants using more general $G$-complexes. For instance, one can take $X$ to be any $(n-1)$-connected CW-complex on which $G$ acts properly, cocompactly, and cellularly (see, e.g. \cite[Definition~1.1]{zaremsky18}); for an even more general condition see \cite[Definition~8.1]{bux04}.
\item[(b)] For any $X$ and $\chi$ as in Definition~\ref{def:BNSR-invariants}, to prove that the filtration $X^{h_\chi\ge t}$ is essentially 
$(m-1)$-connected it suffices to check the corresponding condition just for $t=0$. This follows easily from the fact that $g.X^{h_\chi\ge t} = 
X^{h_\chi\ge {t+\chi(g)}}$ for all $g\in G$ and $t\in \mathbb R$.
\item[(c)] The original definition from \cite{renz_thesis} was as follows: $[\chi]\in\Sigma^m(G)$ if there exists some $X$ as in Definition~\ref{def:BNSR-invariants} such that the subcomplex $X^{h_\chi\ge t}$ is $(m-1)$-connected for all $t\in\mathbb R$. Note that unlike essential $(m-1)$-connectivity, this condition does depend on the choice on $X$. The equivalence of this definition and Definition~\ref{def:BNSR-invariants} was established already in \cite{renz_thesis}; see also \cite[\S~2]{Meinert97} and \cite[\S~2]{llosa_isenrich}.
\end{itemize}
\end{remark}
If $n=1$, a natural choice for $X$ in Definition~\ref{def:BNSR-invariants} is the Cayley graph $\Cay(G,S)$ with respect to some
finite generating set $S$. Likewise, if $n=2$, so that $G$ is finitely presented, we can take $X$ to be the Cayley complex associated to a finite presentation of $G$.

If $n=1$ and $X=\Cay(G,S)$ for a finite generating set $S$, there is no difference between connectivity and essential connectivity (for any $S$);
in fact in this case the following are equivalent for a nonzero character $\chi$:
\begin{itemize}
\item[(a)] $[\chi]\in \Sigma^1(G)$;
\item[(b)] $X^{h_\chi \geq t}$ is connected for all $t$;
\item[(c)] $X^{h_\chi \geq t}$ is connected for some $t$.
\end{itemize}
It is clear that (b) implies (a). That (c) implies (b) is proved, e.g., in \cite[Lemma~A2.9]{strebel_book}. Finally, to prove that (a) implies (c) note that for any $u\le t$, any vertex of $X^{h_\chi \geq u}$ is connected by a path in $X^{h_\chi \geq u}$ to a vertex in $X^{h_\chi \geq t}$ (just choose $s\in S^{\pm 1}$ with $\chi(s)>0$ and multiply by $s^n$ on the right for sufficiently large $n$); hence if $X^{h_\chi \geq t}$ is contained in a connected subset of $X^{h_\chi \geq u}$, then the entire $X^{h_\chi \geq u}$ is connected.

\vskip .1cm

If $G$ is of type $\F_\infty$, then $\Sigma^m(G)$ is defined for all $m$, and we have a nested sequence
\[
\sphere(G) \supseteq \Sigma^1(G) \supseteq \Sigma^2(G) \supseteq\cdots \text{.}
\]
The intersection of all the $\Sigma^m(G)$ is denoted $\Sigma^\infty(G)$. Note that if $G$ is of type $\F$ with a finite classifying space of dimension $n$, then $\Sigma^n(G)=\Sigma^\infty(G)$ (for example this follows from \cite[Citation~1.5]{zaremsky17PB_n}). Thus, for groups of type $\F$ the above sequence stabilizes.

An important property of BNSR-invariants, proved in \cite{bieri87,renz_thesis}, is that $\Sigma^m(G)$ is open in $\sphere(G)$ for all $m<\infty$. Another key property is that $\Sigma^m(G)$ is invariant under the action of $\Aut(G)$, that is, if $\alpha$ is an automorphism of $G$, then for a character $\chi$ we have $[\chi\circ\alpha]\in\Sigma^m(G)$ if and only if $[\chi]\in\Sigma^m(G)$. This easily follows from the fact that $\Sigma^m(G)$ in Definition~\ref{def:BNSR-invariants} does not depend on the choice of $G$-complex $X$ or the choice of $h_\chi$, so in particular one can keep the same complex $X$, but ``twist'' the action by $\alpha$.

\vskip .1cm

The main application of BNSR-invariants is the following theorem of Renz (\cite[Satz~C]{renz_thesis}; see also Theorem~1.1 of \cite{bieri10}).

\begin{theorem}\cite[Satz~C]{renz_thesis}\label{thrm:bnsr_fin_props}
Let $G$ be a group of type $\F_n$ and $H$ a subgroup of $G$ containing the commutator subgroup $[G,G]$. Then $H$ is of type $\F_n$ if and only if $[\chi]\in\Sigma^n(G)$ for all $0\ne \chi \in \Hom(G,\R)$ such that $\chi(H)=0$.
\end{theorem}

For example, $[G,G]$ itself is of type $\F_n$ if and only if $\Sigma^n(G)=\sphere(G)$. At the other extreme in a sense, if $0\ne \chi\colon G\to\R$ is \emph{discrete}, meaning its image is cyclic, then $\ker(\chi)$ is of type $\F_n$ if and only if $[\chi],[-\chi]\in\Sigma^n(G)$. In practice, many groups have BNSR-invariants where one does not need to worry about checking both $\chi$ and $-\chi$, since they cannot have different behavior. Let us write
\[
-\Sigma^m(G)= \{[-\chi]\mid [\chi]\in\Sigma^m(G)\}\text{,}
\]
and make the following definition:

\begin{definition}[Symmetric]\label{def:symmetric}
We say that the BNSR-invariants of a group $G$ are \emph{symmetric} if $\Sigma^m(G)=-\Sigma^m(G)$ for all $m$. For example, if there is an automorphism $\omega$ of $G$ such that $\chi\circ\omega=-\chi$ for all $\chi$, then the BNSR-invariants of $G$ are symmetric.
\end{definition}

\subsection{McCool groups}\label{ssec:groups}

In this subsection we introduce our groups of interest. Let $F_n$ be the free group of rank $n$, with a fixed basis $\{x_1,\dots,x_n\}$. Consider the group $\Aut(F_n)$ of automorphisms of $F_n$. Denote by $\Inn(F_n)$ the subgroup of inner automorphisms, and let $\Out(F_n)=\Aut(F_n)/\Inn(F_n)$ be the group of outer automorphisms. An automorphism of $F_n$ is called \emph{symmetric} if it sends each $x_i$ to a conjugate of some $x_j$, and \emph{pure symmetric} if it sends each $x_i$ to a conjugate of $x_i$. Denote the group of all symmetric automorphisms of $F_n$ by $\SA_n$, and the group of all pure symmetric automorphisms of $F_n$ by $\PSA_n$. Since all inner automorphisms are pure symmetric, we have
\[
\Inn(F_n) \le \PSA_n \le \SA_n \le \Aut(F_n) \text{.}
\]
Taking the quotient by $\Inn(F_n)$, we obtain ``outer'' versions:
\[
\PSO_n\le \SO_n \le \Out(F_n)\text{.}
\]

\medskip

In \cite{mccool86}, McCool exhibited a natural finite presentation for $\PSA_n$. There is a generator $\alpha_{ij}$ for each $1\le i,j\le n$ with $i\ne j$, and the defining relations are as follows:
\begin{itemize}
\item $[\alpha_{ij},\alpha_{k\ell}]=1$ for $i,j,k,\ell$ distinct,
\item $[\alpha_{ij},\alpha_{kj}]=1$ for $i,j,k$ distinct,
\item $[\alpha_{ij}\alpha_{kj},\alpha_{ik}]=1$ for $i,j,k$ distinct.
\end{itemize}
The generator $\alpha_{ij}$ is the automorphism of $F_n$ sending $x_i$ to $x_i^{x_j}=x_j^{-1}x_i x_j$ and fixing all $x_k$ for $k\ne i$. We will generally refer to $\PSA_n$ and $\PSO_n$ as \emph{McCool groups}. Note that $\PSA_1$ is trivial, so when talking about McCool groups we will often implicitly assume $n\ge 2$, which will avoid some pathologies.

Thanks to the second relation, if $I\subseteq \{1,\dots,n\}$ and $j\in \{1,\dots,n\}\setminus I$, then $\alpha_{ij}$ and $\alpha_{i'j}$ commute for all $i,i'\in I$. In particular it makes sense to define
\[
\alpha_{I,j} = \prod\limits_{i\in I}\alpha_{ij} \text{,}
\]
i.e., the order of the product does not matter. The automorphisms $\alpha_{I,j}$ are called the \emph{Whitehead generators} of $\PSA_n$. We say that $\alpha_{I,j}$ is \emph{based} at $j$. Note that $\Inn(F_n)$ is the subgroup of $\PSA_n$ generated by all the $\alpha_{[n]\setminus\{j\},j}$ for $1\le j\le n$, where
\[
[n] = \{1,\dots,n\}\text{.}
\]

\medskip
We will use the following basic relations between the Whitehead generators, which are all clear from the definition:

\begin{lemma}\label{lem:Whiteheadrelations}
The following hold:
\begin{itemize}
\item[(i)] $\alpha_{I,j}\,\alpha_{K,j}=\alpha_{I\cup K,j}$ whenever $I$ and $K$ are disjoint;
\item[(ii)] $[\alpha_{I,j},\alpha_{K,j}]=1$;
\item[(iii)] $[\alpha_{I,j},\alpha_{K,m}]=1$ if either $I\cup \{j\}$ and $K\cup \{m\}$ are disjoint or
$I\cup \{j\}\subseteq K$.\qed
\end{itemize}
\end{lemma}

The McCool groups have some important topological properties that we will need to use later:

\begin{proposition}\label{prop:collins}
The group $\PSA_n$ is torsion-free and admits a free simplicial cocompact action on an $(n-1)$-dimensional contractible simplicial complex. Hence it is of type $\F$. The group $\PSO_n$ has the same properties, with $n-1$ replaced by $n-2$.
\end{proposition}

\begin{proof}
Collins \cite{collins89} proved that the groups $\SA_n$ and $\SO_n$ admit proper simplicial cocompact actions on contractible simplicial complexes of dimensions $n-1$ and $n-2$, respectively. On the other hand, $\PSA_n$ and $\PSO_n$ are torsion-free by a result of Baumslag and Taylor \cite{baumslag68}. Since a proper action of a torsion-free group must be free, restricting the above actions of $\SA_n$ and $\SO_n$ to their finite index subgroups $\PSA_n$ and $\PSO_n$, we obtain free actions which remain cocompact.
\end{proof}

\subsection{Setting the stage for BNSR-invariants of McCool groups}\label{ssec:prior_work}

We can already begin to see some facts relevant to the BNSR-invariants by looking at the McCool presentation. First we can compute abelianizations. Since each defining relation is a commutator, the abelianization of $\PSA_n$ is freely generated by the images of the $\alpha_{ij}$ and thus is isomorphic to $\Z^{n(n-1)}$. Hence $\Hom(\PSA_n,\R)\cong \R^{n(n-1)}$, and so $\sphere(\PSA_n)\cong S^{n(n-1)-1}$. As for $\PSO_n$, in the abelianization we take $\Z^{n(n-1)}$ generated by the images $\overline{\alpha}_{ij}$ of the $\alpha_{ij}$, and we mod out $n$ independent relations, of the form $\overline{\alpha}_{[n]\setminus\{j\},j}=0$ for each $j$. Thus the abelianization of $\PSO_n$ is isomorphic to $\Z^{n(n-1)-n}=\Z^{n(n-2)}$, so $\sphere(\PSO_n)\cong S^{n(n-2)-1}$.

\begin{lemma}\label{lem:symmetric_bnsr}
The BNSR-invariants of $\PSA_n$ and $\PSO_n$ are symmetric.
\end{lemma}

\begin{proof}
For $\PSA_n$ this is \cite[Observation~2.3]{zaremsky18}, and the argument for $\PSO_n$ is analogous. To spell it out, from the McCool presentation we see that there is an automorphism $\omega\in \Aut(\PSA_n)$ sending each $\alpha_{ij}$ to $\alpha_{ij}^{-1}$, which satisfies $\chi\circ \omega = -\chi$ for every character $\chi$ of $\PSA_n$, and the induced automorphism $\overline{\omega}$ of $\PSO_n$ has the same property.
\end{proof}

As we already mentioned, in \cite{orlandikorner00} Orlandi-Korner fully computed the first BNSR-invariant $\Sigma^1(\PSA_n)$ for all $n$. In \cite{zaremsky18}, the second author gave a partial computation of $\Sigma^m(\PSA_n)$ for $m\ge 2$, and in particular proved that all ``positive'' (and all ``negative'') characters of $\PSA_n$ lie in $\Sigma^{n-2}(\PSA_n)\setminus \Sigma^{n-1}(\PSA_n)$.

\begin{remark}
One can also ask about the BNSR-invariants of the full (as opposed to pure) symmetric automorphism groups $\SA_n$ and $\SO_n$, but this turns out to be a much easier question since these groups have very low dimensional character spheres. We have $\Hom(\SA_n,\R)\cong \R$, so $\sphere(\SA_n)=S^0$, and $\Sigma^k(\SA_n)$ is all of $\sphere(\SA_n)$ for $k\le n-2$ and empty for $k\ge n-1$ \cite[Theorem~B]{zaremsky18}. Even more extreme, $\Hom(\SO_n,\R)$ is trivial, since any non-trivial map $\SA_n\to\R$ must restrict non-trivially to $\Inn(F_n)$.
\end{remark}

\section{Empty BNSR-invariants}\label{sec:empty}

In this section we prove parts (b) and (d) of Theorem~\ref{thm:mainPSA}, which assert that certain invariants are empty. This is the easier part; indeed, it was arguably already known to experts, although some pieces of proof need to be assembled.

Let $G$ be a group such that all the rational homology groups $H_i(G,\Q)$ of $G$ are finite-dimensional and vanish for large enough $i$. For example this happens if $G$ is of type $\FP$, meaning the trivial $G$-module $\Z$ admits a finite-length resolution by finitely generated projective $G$-modules. In particular, this happens if $G$ is of type $\F$. The \emph{Euler characteristic} $\euler(G)$ of $G$ is
\[
\euler(G) = \sum\limits_{i=0}^\infty (-1)^i \dim H_i(G,\Q) \text{.}
\]

The following result must be well known, but we could not find it stated in this concise form in the literature; the proof is reasonably elementary.

\begin{proposition}\label{prop:euler}
Let $G$ be a group of type $\F$ such that $\euler(G)\ne 0$. Then $\Sigma^\infty(G)\cap -\Sigma^\infty(G)=\emptyset$, so in particular if the BNSR-invariants of $G$ are symmetric, then $\Sigma^\infty(G)=\emptyset$.
\end{proposition}

\begin{proof}
Suppose $\Sigma^\infty(G)\cap -\Sigma^\infty(G)\ne\emptyset$. Since $G$ is of type $\F$, $\Sigma^\infty(G)=\Sigma^m(G)$ where $m<\infty$ is the dimension of some finite classifying space for $G$, so $\Sigma^m(G)\cap -\Sigma^m(G)\ne\emptyset$. Since $\Sigma^m(G)$ is open in $\sphere(G)$, we can choose a discrete character $0\ne\chi\colon G\to\Z$ such that $[\pm\chi]\in\Sigma^m(G)=\Sigma^\infty(G)$. Let $K=\ker(\chi)$, so $K$ is of type $\F_\infty$. 

Since $G$ has a finite classifying space of dimension $m$, its cohomological dimension $cd(G)$ is at most $m$ and hence 
$cd(K)\leq cd(G)<\infty$ (see, e.g., \cite[Propositions~VIII.2.2,~VIII.2.4]{brown82}). Combined with $K$ being of type $\F_\infty$, hence $\FP_\infty$, this implies that $K$ is of type $\FP$ by \cite[Proposition~VIII.6.1]{brown82}. Now by \cite[Proposition~IX.7.3(d)]{brown82}, since $K$, $G$, and $\Z$ are all of type $\FP$, the Euler characteristic $\euler$ is multiplicative for the short exact sequence $1\to K\to G\to\Z\to 1$, and we get $\euler(G)=\euler(K)\cdot\euler(\Z)=\euler(K)\cdot 0=0$, a contradiction.
\end{proof}

\begin{remark}
The above proof applies equally well to the homological BNSR-invariant $\Sigma^\infty(G;\Z)$, since we only need to know that $K$ was of type $\FP_\infty$ (not necessarily $\F_\infty$). We should also mention that, as explained to us by Dawid Kielak, using some deep results from $\ell^2$-homology the assumption about the BNSR-invariants being symmetric can be relaxed, and the conclusion can be strengthened to saying that $\Sigma^\infty(G;\Q)=\emptyset$, where $\Sigma^\infty(G;\Q)$ is the homological BNSR-invariants with $\Q$ coefficients. This follows from Theorem~A of \cite{hughes_kielak}, together with the fact that for groups of type $\F$ the Euler characteristic and $\ell^2$-Euler characteristic coincide (see \cite[Theorem~6.80(1)]{lueck02}). For our purposes, we only need the result as stated in Proposition~\ref{prop:euler}.
\end{remark}

\medskip

We can now prove Theorem~\ref{thm:mainPSA}(b)(d), asserting that $\Sigma^{n-1}(\PSA_n)=\emptyset$ and $\Sigma^{n-2}(\PSO_n)=\emptyset$.

\begin{proof}[Proof of Theorem~\ref{thm:mainPSA}(b)(d)]
By Proposition~\ref{prop:collins}, $\PSA_n$ and $\PSO_n$ are of type $\F$. By Lemma~\ref{lem:symmetric_bnsr}, the BNSR-invariants of these groups are symmetric. Finally, it is known that $\euler(\PSA_n)\ne 0$ and $\euler(\PSO_n)\ne 0$; in fact, $\euler(\PSA_n)=(1-n)^{n-1} \ne 0$ and $\euler(\PSO_n)=(1-n)^{n-2}\ne 0$; see Theorems~1.2 and~8.2 of \cite{mccammond04}. Thus,
$\Sigma^{n-1}(\PSA_n)=\emptyset$ and $\Sigma^{n-2}(\PSO_n)=\emptyset$ by Proposition~\ref{prop:euler}.
\end{proof}

\section{Dense BNSR-invariants}\label{sec:dense}

In this section we prove the much more difficult, ``positive'' statements of Theorem~\ref{thm:mainPSA}, parts (a) and (c), asserting that $\Sigma^{n-2}(\PSA_n)$ and $\Sigma^{n-3}(\PSO_n)$ are dense in their respective character spheres. The general plan is to apply a criterion of Meier, Meinert and Van Wyk, with the bulk of the work relating to the  Whitehead poset $\WO_n$ and McCullough--Miller space $\MMO_n$ of $\PSO_n$, along with natural ``auter'' variants. In the coming subsections, we formulate the Meier--Meinert--Van Wyk criterion, discuss the Whitehead poset and McCullough--Miller space, and then prove the density results.

\subsection{Higher generation and the Meier--Meinert--Van Wyk criterion}\label{ssec:mmvk_criterion}

The key to proving that certain characters are in certain BNSR-invariants will be a general criterion due to Meier, Meinert, and Van Wyk~\cite{meier98} (see Theorem~\ref{thrm:MMV98} below), involving the concept of higher generation
due to Abels and Holz. We start with several definitions.

\smallskip
 
Let $G$ be a group, let $\calH$ be a family of subgroups of $G$, and let $\Cos(G,\calH)$ be the collection of all cosets $gH$ with $g\in G$ and $H\in\calH$. We will view both $\calH$ and $\Cos(G,\calH)$ as posets with respect to inclusion. Recall that given a poset $P$, the associated flag complex $\calF(P)$ is the simplicial complex whose simplices are finite chains of elements of $P$ (this is also known as the order complex). If $P$ has a global minimum element $\widehat{0}$, let us also write
\[
\calF_0(P)= \calF(P\setminus\{\widehat{0}\})\text{.}
\]
Below we will be concerned with two flag complexes associated to $G$ and $\calH$. The first is $\calF(\Cos(G,\calH))$, the flag complex of the poset of all cosets of subgroups in $\calH$. The second is $\calF_0(\calH)$, the flag complex of the poset of all non-trivial subgroups in $\calH$ (only the non-trivial subgroups, not all cosets thereof).

Let $\calU=\{U_i\mid i\in I\}$ be a covering of a set $X$. Recall that the \emph{nerve} of $\calU$, denoted by $N(\calU)$, is the simplicial complex with vertex set $I$ such that $i_0,\dots,i_p$ span a $p$-simplex whenever $U_{i_0}\cap\cdots\cap U_{i_p}$ is non-empty.

The following definition was introduced by Abels and Holz in \cite{abels93}.

\begin{definition}[$n$-generating]\label{def:hi_gen}
Let $\calH$ be family of subgroups of a group $G$, so $\Cos(G,\calH)$ is a covering of $G$. For $n\in\N$, the family $\calH$ is said to be \emph{$n$-generating} for $G$ if the nerve $N(\Cos(G,\calH))$ is $(n-1)$-connected, and 
\emph{$\infty$-generating} if the nerve is contractible.
\end{definition}

\begin{remark}\label{rmk:flag}
If $\calH$ is closed under finite intersections, then by \cite[Theorem~1.4(b)]{abels93}, $N(\Cos(G,\calH))$ is homotopy equivalent to the flag complex $\calF(\Cos(G,\calH))$. Thus, in this case we can replace $N(\Cos(G,\calH))$ in Definition~\ref{def:hi_gen} with the complex $\calF(\Cos(G,\calH))$, which is somewhat easier to understand, being the flag complex of a poset.
\end{remark}

We are now ready to state the general result of Meier--Meinert--Van Wyk from \cite{meier98} relating BNSR-invariants to highly generating families of subgroups:

\begin{theorem}\cite[Theorem~3.1]{meier98}\label{thrm:MMV98}
Let $G$ be a group of type $\F_n$. Let $\calH$ be a non-empty finite intersection-closed $n$-generating family of subgroups of $G$ such that each $H\in\calH$ is of type $\F_n$. Let $\chi\colon G \to\R$ be a character such that $\chi|_H\ne 0$ for each non-trivial subgroup $H\in \calH$. Suppose that the following hold:
\begin{itemize}
    \item[(i)] Either the trivial subgroup does not belong to $\calH$ or the simplicial complex $\calF_0(\calH)$ (as defined above) is $(n-1)$-connected, and
    \item[(ii)] $[\chi|_H]\in \Sigma^n(H)$ for all non-trivial $H\in\calH$.
\end{itemize}
Then $[\chi]\in\Sigma^n(G)$.
\end{theorem}

\begin{remark}
The condition that the trivial subgroup does not belong to $\calH$ is a strong constraint since $\calH$ needs to be closed under intersections. On the other hand, if the trivial subgroup does belong to $\calH$, the complex $\calF(\calH)$ is trivially contractible, which explains why we are looking at $\calF_0(\calH)$ rather than $\calF(\calH)$ in this case.
\end{remark}

\subsection{The Whitehead poset}\label{ssec:whitehead_poset}

In this subsection we recall the definition and prove some properties of the Whitehead poset $\WO_n$, generally following \cite{brady01}. The elements of $\WO_n$ can be viewed combinatorially as bipartite labeled trees on $[n] = \{1,2,\dots,n\}$, which we now define.

\begin{definition}
A \emph{bipartite labeled tree on $[n]$} is a tree that has $n$ vertices labeled by the integers from $1$ to $n$ and some finite number of unlabeled vertices, satisfying the following conditions:
\begin{itemize}
    \item[(i)] each edge connects a labeled vertex to an unlabeled vertex;
    \item[(ii)] each unlabeled vertex has degree at least $2$.
\end{itemize}
\end{definition}

The set of all bipartite labeled trees on $[n]$ will be denoted by $\WO_n$. We will now define a partial order on $\WO_n$. If $T\in \WO_n$ has a labeled vertex $v$ connected by edges to two unlabeled vertices $w_1$ and $w_2$, we can form a new tree $T'$ by identifying $w_1$ and $w_2$ and identifying the edges joining $w_1$ and $w_2$ to $v$. In this case we say that $T'$ is obtained from $T$ by a \emph{folding} at $v$. See Figure~\ref{fig:folding} for an example.

\begin{figure}[htb]
\centering
\begin{tikzpicture}[line width=1pt]

\draw (0.2,0.2) -- (1,1) -- (0.2,1.8)   (1,1) -- (4,1);
\filldraw (0.2,0.2) circle (1.5pt);
\filldraw (1,1) circle (1.5pt);
\filldraw (0.2,1.8) circle (1.5pt);
\filldraw (3,1) circle (1.5pt);
\filldraw (4,1) circle (1.5pt);
\filldraw (2,1) circle (1.5pt);

\node at (-0.1,0.2) {$1$};
\node at (-0.1,1.8) {$2$};
\node at (2,1.3) {$3$};
\node at (4.3,1) {$4$};

\node at (6,1) {$\longrightarrow$};

\begin{scope}[xshift=8cm, yshift=.25cm]
\draw (0,0) -- (1.5,1.5)   (0,1.5) -- (1.5,0);

\filldraw (0,0) circle (1.5pt);
\filldraw (1.5,1.5) circle (1.5pt);
\filldraw (0,1.5) circle (1.5pt);
\filldraw (1.5,0) circle (1.5pt);
\filldraw (.75,.75) circle (1.5pt);

\node at (-0.3,0) {$1$};
\node at (-0.3,1.5) {$2$};
\node at (1.8,1.5) {$3$};
\node at (1.8,0) {$4$};
\end{scope}

\end{tikzpicture}
\caption{Bipartite labeled trees on $[4]$, with the second tree obtained from the first tree by a folding at the vertex labeled $3$.}
\label{fig:folding}
\end{figure}

Define a partial order $\preceq$ on $\WO_n$ as follows: given $T_1,T_2\in \WO_n$, we set $T_1\preceq T_2$ if $T_1$ can be obtained from $T_2$ by a (possibly empty) sequence of foldings.

\begin{definition}\label{def:whitehead}
The set $\WO_n$ with the partial order $\preceq$ is called the \emph{Whitehead poset}.
\end{definition}

We will also be working with $\WA_n$, the ``auter'' version of the Whitehead poset. We define $\WA_n$ to be the subposet of $\WO_{n+1}$ consisting of all trees $T$ for which the vertex labeled $n+1$ has degree $1$.

If $T\in \WO_n$ has $k$ unlabeled vertices, define the \emph{degree} of $T$ to be
\[
\deg(T) = k-1\text{.}
\]
Let us collect a few basic properties of foldings and the degree function on $\WO_n$.

\begin{observation}\label{obs:degreebasic}
The following hold:
\begin{itemize}
\item[(a)] If $T'$ is obtained from $T$ by a folding, then $\deg(T')=\deg(T)-1$.
\item[(b)] $0\le \deg(T)\le n-2$ for all $T\in \WO_n$.
\item[(c)] Every tree of positive degree admits at least one folding.
\item[(d)] Every tree of degree $< n-2$ admits at least one unfolding (the reverse of a folding).
\item[(e)] $\WO_n$ contains a unique tree of degree $0$ (in which the unique unlabeled vertex is connected to all $n$ labeled vertices), and this tree is the (unique) global minimum of $\WO_n$.
\item[(f)] The maximal elements of $\WO_n$ are precisely the trees of degree $n-2$. Moreover, for every maximal tree $T$ there exists a chain $T_0\prec T_1\prec \cdots\prec T_{n-2}=T$ in $WO_n$.
\item[(g)] The maximal elements of $\WA_{n-1}$  are precisely the trees of degree $n-2$ which lie in $\WA_{n-1}$ (recall that $\WA_{n-1}$ is defined as a subset of $\WO_n$).  Moreover, for every maximal tree $T\in WA_{n-1}$, every chain as in (f) is contained in $\WA_{n-1}$.
\end{itemize}
\end{observation}

\begin{proof}
(a) This is clear from the definition of folding.

(b) Suppose that $T$ has $k$ unlabeled vertices (that is, $k=\deg(T)+1$). Since every edge connects an unlabeled vertex to a labeled vertex, we must have $k\geq 1$, so $\deg(T)\geq 0$. The total number of vertices in $T$ is $n+k$ and hence (since $T$ is a tree), it has exactly $n+k-1$ edges. Since in addition every unlabeled vertex has degree $\geq 2$, we have $n+k-1\geq 2k$, that is $k\leq n-1$, or equivalently $\deg(T)\leq n-2$, which proves (b).

(c) A tree admits no foldings if and only if all labeled vertices have degree $1$. This happens exactly when $n+k-1$ (the total number of edges) equals $n$ (the total number of labeled vertices), that is, when $k=1$ or, equivalently, $\deg(T)=0$.

(d) A tree $T$ admits an unfolding if and only if it has an unlabeled vertex of degree $>2$, and by the computation in (b) the latter happens exactly when $\deg(T)<n-2$.

(e) The first assertion is clear and the second one follows from (c).

(f) If $\deg(T)=n-2$, then $T$ is maximal by (a) and (b), and if $\deg(T)<n-2$, then $T$ is not maximal by (d). The last assertion of (f) follows from (c) and (a).

(g) As in (f), if $\deg(T)=n-2$, then $T$ is maximal. If $\deg(T)<n-2$, then by the proof of~(d) $T$ has an unlabeled vertex of degree $>2$, call it $v$. At least one of the neighbors of $v$ is labeled by some $i\neq n+1$, and an unfolding at that neighbor produces a tree $T'\in \WA_{n-1}$ with $T\prec T'$, so $T$ is not maximal. The last assertion of (g) holds simply because foldings do not increase the degrees of labeled vertices, so if $T\in \WA_{n-1}$, then $S\in \WA_{n-1}$ for any $S\prec T$.
\end{proof}

The following result is an immediate consequence of parts (f) and (g) of Observation~\ref{obs:degreebasic}.

\begin{corollary}\label{prop:dim}
The flag complex $\calF(\WO_n)$ is $(n-2)$-dimensional, and the flag complex $\calF(\WA_n)$ is $(n-1)$-dimensional.
\end{corollary}

\subsection{A family of abelian subgroups}\label{ssec:subgroups}

The main goal of this subsection is to show that the Whitehead poset is isomorphic to a certain poset of abelian subgroups of $\PSO_n$ ordered by inclusion. To build up to this, let us construct a subgroup $H(T)\le \PSO_n$ associated to each $T\in \WO_n$. Given $T\in \WO_n$ and $j\in [n]$, let $T_j$ be the graph obtained from $T$ by removing the vertex labeled $j$ (and all edges incident to it), and consider the partition of $[n]\setminus \{j\}$ where two elements lie in the same block if and only if the vertices of $T_j$ with the corresponding labels lie in the same connected component. Denote this partition by $\calP(T,j)$. Now let
\[
H(T) = \bigg\langle \alpha_{I,j} ~\bigg|~ j\in[n]\text{, } I\in \calP(T,j)\bigg\rangle/\Inn(F_n) \text{,}
\]
and write
\[
\mathcal{HO}_n = \{H(T) \mid T\in\WO_n\}\text{.}
\]

\smallskip

\begin{proposition}\label{lem:MMabelian}
The group $H(T)$ is free abelian of rank $\deg(T)$.
\end{proposition}

Proposition~\ref{lem:MMabelian} is easy to deduce from \cite[Proposition~5.1]{mccullough96} (see also  \cite[Theorem~2.3(ii)]{jensen06}); however, since the groups  $H(T)$ are introduced in \cite{mccullough96} in a different way and moreover \cite{mccullough96} considers symmetric automorphisms of arbitrary free products (as opposed to free groups in our paper) making the statements more technical, we will provide a direct proof for the convenience of the reader.

Below by $\overline{\alpha}_{i,j}$ we denote the image of $\alpha_{i,j}$ in $\PSO_n$. We first establish an auxiliary result.

\begin{lemma}\label{lem:MMaux}
For each $1\leq j\leq n$ let $G_j=\langle \overline{\alpha}_{i,j}\mid i\neq j\rangle$. The following hold:
\begin{itemize}
\item[(a)] The group $G_j$ is abelian.
\item[(b)] Fix $1\le j\le n$. For $c_i\in\Z$, the product $\prod\limits_{j\neq i}\overline{\alpha}_{i,j}^{c_{i}}$ equals $1$ if and only if all $c_i$ are the same. Hence $G_j\cong\mathbb Z^{n-2}$.
\item[(c)] Let $g=\prod\limits_{j=1}^n g_j$ with $g_j\in G_j$ for each $j$. Then $g=1$ if and only if each $g_j=1$.
\end{itemize}
\end{lemma}

\begin{proof}
(a) This is immediate from Lemma~\ref{lem:Whiteheadrelations}(i). 

(b) The ``if'' part is clear. The ``only if'' part holds since an automorphism of the form  $\prod\limits_{j\neq i}\alpha_{i,j}^{c_{i}}$ always fixes $x_j$, so it can only be inner if it is the conjugation by a power of $x_j$.

(c) Let $V=\PSO_n/[\PSO_n,\PSO_n]$, and let $\pi: \PSO_n\to V$ be the natural projection. As we already explained in Subsection~\ref{ssec:prior_work}, 
$V\cong \mathbb Z^{n(n-2)}$. Since $V=\sum_{j=1}^n \pi(G_j)$ and each $G_j\cong\mathbb Z^{n-2}$ by~(b), it follows that $\pi$ is injective on each $G_j$ and $V=\oplus_{j=1}^n \pi(G_j)$. This implies~(c).
\end{proof}

\begin{proof}[Proof of Proposition~\ref{lem:MMabelian}]
We first prove that $H(T)$ is abelian. Take any two generators $\alpha_{I,j}$ and $\alpha_{K,m}$ of $H(T)$, so $I\in \calP(T,j)$ and $K\in \calP(T,m)$. If $j=m$, these generators commute by Lemma~\ref{lem:Whiteheadrelations}(ii).

Suppose now that $j\neq m$. Since $\overline{\alpha}_{I,j}^{\,-1} = \overline{\alpha}_{[n]\setminus (I\cup\{j\}),j}$, replacing $\alpha_{I,j}$ by $\alpha_{[n]\setminus (I\cup\{j\}),j}$ if needed, we can assume that $m\not\in I$ (and hence $m\not\in I\cup\{j\}$). Since $I\in \calP(T,j)$, by definition of the partition $\calP(T,j)$, there exists a subtree of $T$ whose set of vertex labels is precisely $I\cup\{j\}$. Since $m\not\in I\cup\{j\}$, the set $I\cup\{j\}$ must lie in a single block of the partition $\calP(T,m)$. Hence either $I\cup\{j\}\subseteq K$ or $I\cup\{j\}$ is disjoint from $K$ (and hence from $K\cup\{m\}$). In each case $\alpha_{I,j}$ and $\alpha_{K,m}$ commute by Lemma~\ref{lem:Whiteheadrelations}(iii). 
\vskip .1cm

Thus, we have proved that $H(T)$ is abelian, so any element of $H(T)$ can be written as a product $\prod\limits_{j=1}^n p_j$ where each $p_j=\prod\limits_{I\in \calP(T,j)}\overline{\alpha}_{I,j}^{c_{I,j}}$ for some $c_{I,j}\in\mathbb Z$. By Lemma~\ref{lem:MMaux}(c), $\prod\limits_{j=1}^n p_j=1$ if and only if $p_j=1$ for each $1\leq j\leq n$. For a fixed $j$,
the elements of $\calP(T,j)$ are pairwise disjoint, so by Lemma~\ref{lem:MMaux}(b) and Lemma~\ref{lem:Whiteheadrelations}(i) we have $p_j=1$ if and only if the integers $c_{I,j}$ are the same  for all $I$. Hence $H(T)$ is free abelian of rank $\sum\limits_{j=1}^n (d_j-1)=\sum\limits_{j=1}^n d_j-n$ where $d_j$ is the degree of the vertex of $T$ labeled $j$.

Since every edge connects a labeled vertex to an unlabeled vertex, $\sum\limits_{j=1}^n d_j$ is the total number of edges of $T$, which is also equal to $n+\deg(T)$ since $T$ has $n$ labeled vertices and $\deg(T)+1$ unlabeled vertices. Thus, the rank of $H(T)$ is equal to $(n+\deg(T))-n=\deg(T)$, as desired.
\end{proof}

In the ``auter'' case, for each $T\in\WA_n$ let
\[
H_A(T)= H(T) \cap \PSA_n \text{,}
\]
where we view $\PSA_n$ as a subgroup of $\PSO_{n+1}$ by having each automorphism act trivially on $x_{n+1}$ (and noting that non-trivial  automorphisms of this form are never inner). Let
\[
\mathcal{HA}_n= \{H_A(T)\mid T\in\WA_n\}\text{.}
\]
Note that each $H_A(T)$ is abelian, being a subgroup of $H(T)$.

\smallskip

We will show that the map $T\mapsto H(T)$ is an isomorphisms of posets between $\WO_n$ and $\mathcal{HO}_n$ (ordered by inclusion) and likewise the map $T\mapsto H_A(T)$ is an isomorphism of posets between $\WA_n$ and $\mathcal{HA}_n$. First we need some preparations.

\begin{definition}
Let $T\in \WO_n$ and $j\in [n]$. A subset $I$ of $[n]\setminus\{j\}$ will be called \emph{$(T,j)$-complete} if $I$ is a (possibly empty) union of some blocks of the partition $\calP(T,j)$.
\end{definition}

\begin{lemma}\label{lem:Tjcomplete}
Let $T\in WO_n$ and $I\subseteq [n]\setminus\{j\}$. Then 
$\overline{\alpha}_{I,j}\in H(T)$ if and only if $I$ is $(T,j)$-complete.
\end{lemma}

\begin{proof} If $I$ is $(T,j)$-complete, then $I=\sqcup_k I_k$ where each $I_k\in\calP(T,j)$ and hence
$\overline{\alpha}_{I,j}=\prod_k \overline{\alpha}_{I_k,j}\in H(T)$.

Let us now prove the converse. Suppose that $\overline{\alpha}_{I,j}\in H(T)$ for some $j\in [n]$ and $I\subseteq [n]\setminus\{j\}$, so that
\[
\overline{\alpha}_{I,j}=\prod_{t=1}^k (\overline{\alpha}_{B_t,j_t})^{c_t},
\]
where $B_t\in \calP(T,j_t)$ for each $t$, all pairs $(B_t,j_t)$ are distinct, and all $c_t\neq 0$. By Lemma~\ref{lem:MMaux}(c), we can assume that $j_t=j$ for each $t$ in the above equality, and we can rewrite the latter as
\[
\prod_{i\in I}\overline{\alpha}_{i,j} = \prod_{t=1}^k 
\left(\prod_{i\in B_t}\overline{\alpha}_{i,j}\right)^{c_t}. \eqno (***)
\]
Suppose now that some $i,i'\in [n]\setminus\{j\}$ lie in the same block of the partition $\calP(T,j)$. Then $\overline{\alpha}_{i,j}$ and $\overline{\alpha}_{i',j}$ appear on the right-hand side of (***) with the same exponent. Since the only vanishing products of the elements $\overline{\alpha}_{i,j}$ for a fixed $j$ are powers of $\prod_{i\in [n]\setminus\{j\}}\overline{\alpha}_{i,j}$, it follows that $\overline{\alpha}_{i,j}$ and $\overline{\alpha}_{i',j}$ also have the same exponent on the left-hand side of (***), so $i$ and $i'$ either both lie in $I$ or both lie outside of $I$. It follows that $I$ is a union of blocks of $\calP(T,j)$, so it is $(T,j)$-complete.
\end{proof}

\begin{lemma}\label{lem:bij}
The function $\eta\colon T\mapsto H(T)$ is a bijection from $\WO_n$ to $\mathcal{HO}_n$, and the function $\eta_A\colon T\mapsto H_A(T)$ is a bijection from $\WA_n$ to $ \mathcal{HA}_n$.
\end{lemma}

\begin{proof}
The functions $\eta$ and $\eta_A$ are surjective by construction, so we just need to prove injectivity. 

We first prove that $\eta$ is injective by describing an algorithm
for extracting $T$ from $H(T)$. Fix $T\in \WO_n$. We claim that for any distinct $i,j\in [n]$ the following conditions are equivalent:
\begin{itemize}
\item[(i)] The vertices labeled $i$ and $j$ are connected to the same unlabeled vertex in $T$.
\item[(ii)] The unique reduced path in $T$ connecting the vertices labeled
$i$ and $j$ contains no other labeled vertices.
\item[(iii)] The labels $i$ and $j$ lie in the same block of the partition $\calP(T,k)$ for any $k\in [n]\setminus\{i,j\}$.
\item[(iv)] For any $\overline{\alpha}_{I,k}\in H(T)$ with $k\neq i,j$, the set $I$ contains either both $i$ and $j$ or neither of them.
\end{itemize}
The equivalences (i)$\iff$(ii) and (ii)$\iff$(iii) are clear and 
(iii)$\iff$(iv) holds by Lemma~\ref{lem:Tjcomplete}. Since condition (iv) can be read off from $H(T)$, the same is true for (i).

Let us now form an auxiliary graph $\mathrm{Aux}(T)$. It has $n$ vertices, labeled by the elements of $[n]$, and no unlabeled vertices; the vertices labeled by $i$ and $j$ are connected by an edge if and only if the vertices of $T$ labeled $i$ and $j$ are connected by edges to the same unlabeled vertex in $T$. By the previous paragraph, $\mathrm{Aux}(T)$ is determined by $H(T)$. On the other hand, it is easy to recover $T$ from $\mathrm{Aux}(T)$: create an unlabeled vertex for each maximal clique $C$ in $\mathrm{Aux}(T)$ and connect that vertex to a vertex labeled by $i$ for each $i\in C$. Thus, we have proved that $\eta$ is injective.

We now deduce injectivity of $\eta_A$ from that of $\eta$. For this, we need to prove that if $T,U\in \WA_n$ are such that $H(T)\cap\PSA_n=H(U)\cap\PSA_n$ then $H(T)=H(U)$. By symmetry it suffices to check that $H(T)\subseteq H(U)$.
Consider any $\overline{\alpha}_{I,j}\in H(T)$, so $j\in [n+1]$ and $I$ is $(T,j)$-complete. We will prove that $\overline{\alpha}_{I,j}\in H(U)$ by splitting into three cases.

\textbf{Case 1:} Suppose $j\in[n]$ and $n+1\not\in I$. Then $\overline{\alpha}_{I,j}\in \PSA_n$, hence the hypothesis $H(T)\cap\PSA_n=H(U)\cap\PSA_n$ ensures that $\overline{\alpha}_{I,j}\in H(U)$. 

\textbf{Case 2:} Suppose $j=n+1$. Since $n+1$ is a leaf in $T$, we have $I=[n]$ or $\emptyset$; in either case $\alpha_{I,j}$ is inner, so $\overline{\alpha}_{I,j}=1$ is automatically in $H(U)$.

\textbf{Case 3:} Suppose $j\in[n]$ and $n+1\in I$. Let $I'$ be the complement of $I$ in $[n+1]\setminus \{j\}$, so $\alpha_{I,j}\alpha_{I',j}$ is inner. Hence
$\overline{\alpha}_{I',j} = \overline{\alpha}_{I,j}^{\,\,-1}\in H(T)$, but now $n+1\not\in I'$, so by case~1 we have $\overline{\alpha}_{I',j}\in H(U)$, and thus $\overline{\alpha}_{I,j}\in H(U)$ as desired.
\end{proof}

Recall that a poset is a \emph{meet semilattice} if any two elements $x$ and $y$ have a unique greatest lower bound, called their \emph{meet}, denoted $x\wedge y$.

\begin{proposition}\cite[Lemmas~5.4 and~5.5]{mccullough96}\label{prop:int_closed}
The posets $\WO_n$ and $\WA_n$ are meet semilattices. Moreover, $H(T\wedge U)=H(T)\cap H(U)$ for any $T,U\in\WO_n$ and $H_A(T\wedge U)=H_A(T)\cap H_A(U)$ for any $T,U\in\WA_n$. In particular, the families $\mathcal{HA}_n$ and $\mathcal{HO}_n$ are closed under taking intersections and so are also meet semilattices.
\end{proposition}

Now let us prove that the bijection $\eta\colon T\mapsto H(T)$ is a poset isomorphism from $\WO_n$ to $\mathcal{HO}_n$, along with the ``auter'' version.

\begin{proposition}\label{prop:poset_iso}
For $T,U\in\WO_n$, we have $T\preceq U$ if and only if $H(T)\le H(U)$. For $T,U\in\WA_n$, we have $T\preceq U$ if and only if $H_A(T)\le H_A(U)$. Hence $\eta$ and $\eta_A$ are isomorphisms of posets, and we get $\calF(\WO_n)\cong \calF(\mathcal{HO}_n)$ and $\calF(\WA_n)\cong \calF(\mathcal{HA}_n)$. Analogous results also hold with $\calF$ replaced by $\calF_0$ everywhere.
\end{proposition}

\begin{proof}
Suppose that $T\preceq U$, so $T$ is obtained from $U$ by a sequence of foldings. Then for any $j\in [n]$ and any $I\subseteq [n]\setminus \{j\}$ that is $(T,j)$-complete, $I$ is also $(U,j)$-complete. We conclude that $H(T)\le H(U)$.

Now suppose that $H(T)\le H(U)$. This implies that $H(T)=H(T)\cap H(U)$, and so by Proposition~\ref{prop:int_closed} we have $H(T)=H(T\wedge U)$. By Lemma~\ref{lem:bij}, this tells us that $T=T\wedge U$, and so $T\preceq U$. Since we already know that $\eta:\WO_n\to \mathcal{HO}_n$ is a bijection, this proves that it is an isomorphism of posets.

The proof for $\eta_A$ is analogous, and the remaining statements in the proposition are all immediate consequences of the facts that $\eta$ and $\eta_A$ are poset isomorphisms.
\end{proof}

\subsection{McCullough--Miller space}\label{ssec:mm_space}

We are now ready to define the McCullough--Miller spaces $\MMO_n$ and $\MMA_n$, slightly rephrasing the original definitions from \cite{mccullough96} and \cite{chen05}. The idea for $\MMO_n$ is to glue together copies of $\calF(\WO_n)$, one for each element of $\PSO_n$, along appropriate intersections. Since we have already established that $\WO_n$ is isomorphic to $\mathcal{HO}_n$, it is not surprising that the end result of this is a complex isomorphic to $\calF(\Cos(\PSO_n,\mathcal{HO}_n))$, and indeed we will simply define McCullough--Miller space this way. The equivalence between this viewpoint and the original viewpoint of McCullough--Miller in \cite{mccullough96} is not difficult, and is discussed for example in \cite{griffin13} (along with the $\MMA_n$ case).

\begin{definition}[McCullough--Miller space]\label{def:mm}
Define the \emph{McCullough--Miller space} $\MMO_n$ to be
\[
\MMO_n = \calF(\Cos(\PSO_n,\mathcal{HO}_n))\text{,}
\]
and define the ``auter'' version $\MMA_n$ to be
\[
\MMA_n = \calF(\Cos(\PSA_n,\mathcal{HA}_n))\text{.}
\]
\end{definition}

The only property of the McCullough--Miller spaces that we will need is also the most important one, namely, that they are contractible:

\begin{proposition}\label{prop:contractible}
The complexes $\MMO_n$ and $\MMA_n$ are contractible. Thus, $\mathcal{HO}_n$ and $\mathcal{HA}_n$ are $\infty$-generating, for $\PSO_n$ and $\PSA_n$ respectively.
\end{proposition}

\begin{proof}
Contractibility of $\MMO_n$ was proved by McCullough and Miller in \cite[Section~4]{mccullough96}. Contractibility of $\MMA_n$ is essentially proved by Chen, Glover and Jensen in \cite[Theorem~3.1]{chen05}. Technically, in \cite{chen05}, Chen, Glover, and Jensen consider pure symmetric automorphisms not of free groups, but of free products of finite groups. However, the proof of contractibility of the associated complex works for free products of any groups, in particular free groups, as explained by Griffin in \cite[Section~5]{griffin13}; see \cite[Theorem~5.5]{griffin13}. The statement about $\infty$-generation is now immediate from Remark~\ref{rmk:flag}.
\end{proof}

\subsection{Proof of density}\label{ssec:completion}

In this subsection we prove parts (a) and (c) of Theorem~\ref{thm:mainPSA}, asserting that $\Sigma^{n-2}(\PSA_n)$ and $\Sigma^{n-3}(\PSO_n)$ are dense in their respective character spheres. The idea is to prove that a suitably generic character satisfies the hypotheses of Theorem~\ref{thrm:MMV98}, using the family of subgroups $\mathcal{HA}_n$ for $\PSA_n$ and $\mathcal{HO}_n$ for $\PSO_n$.

Most of the hypotheses involved in Theorem~\ref{thrm:MMV98} have been dealt with by now. The one we still need to prove is that the flag complexes of non-trivial subgroups from the above families are highly connected. We will actually prove a stronger statement, involving the property of being homotopy Cohen--Macualay:

\begin{definition}[Homotopy Cohen--Macaulay]
A simplicial complex $X$ of finite dimension $n$ is called \emph{homotopy Cohen--Macaulay} if for every $p$-simplex $\sigma$, the link $\mathrm{lk}(\sigma)$ is $(n-p-2)$-connected. (This includes the empty simplex, so $X$ itself should be $(n-1)$-connected.) A poset is called \emph{homotopy Cohen--Macaulay} if its flag complex is.
\end{definition}

\begin{proposition}\label{prop:hi_conn}
The posets $\WO_n$ and $\WA_n$ are homotopy Cohen--Macaulay. Hence, the flag complex $\calF_0(\mathcal{HO}_n)$ is $(n-4)$-connected and $\calF_0(\mathcal{HA}_n)$ is $(n-3)$-connected.
\end{proposition}

\begin{remark}Since Proposition~\ref{prop:dim} says that the flag complexes of $\WO_n$ and $\WA_n$ are $(n-2)$-dimensional and $(n-1)$-dimensional, respectively, and since $\calF_0(\mathcal{HO}_n)$ and $\calF_0(\mathcal{HA}_n)$ are isomorphic to links of $0$-simplices in the flag complexes of $\WO_n$ and $\WA_n$, respectively, namely the trivial vertices, the second statement follows from the first.
\end{remark}

The fact that $\WO_n$ is homotopy Cohen--Macaulay has already been proved in \cite[Theorem~5.13]{brady01}. To show that $\WA_n$ is also homotopy Cohen--Macaulay, we need to recreate and adjust the proof for $\WO_n$. First we will introduce some terminology dealing with posets. 

\begin{definition}
A poset $P$ is called 
\begin{itemize}
\item \emph{bounded} if it has a global miminum and a global maximum;
\item \emph{graded} if it is bounded and all maximal chains have the same length, which is then called the \emph{length} of $P$. 
\end{itemize}
\end{definition}

Given $p\leq q$ in a poset $P$, write $[p,q]=\{r\mid p\le r\le q\}$, and call this a (closed) \emph{interval}. Given $p,q\in P$ with $p\ne q$, we say that $q$ \emph{covers} $p$ if $[p,q]=\{p,q\}$. An \emph{atom} in a poset with a global minimum is any element that covers the global minimum.

There is a nice sufficient condition (see \cite[Lemma~5.3]{brady01}), originally due to Bj\"orner and Wachs \cite{bjoerner83}, for a bounded graded poset $P$ to be homotopy Cohen--Macaulay, namely that it admits a recursive atom ordering:  

\begin{definition}
Let $P$ be a bounded graded poset with global minimum $\widehat{0}$ and global maximum $\widehat{1}$. We say that $P$ \emph{admits a recursive atom ordering} if either $P$ has length $1$ or $P$ has length greater than $1$ and there is an ordering $a_1,\dots,a_\ell$ of the atoms of $P$ such that the following hold:
\begin{enumerate}
    \item For each $1\le j\le \ell$, the interval $[a_j,\widehat{1}]$ admits a recursive atom ordering such that any atoms of $[a_j,\widehat{1}]$ that cover some $a_i$ with $i<j$ come first in the ordering.
    \item For all $i<j$, if $a_i,a_j < y$ then there exist $i'<j$ and $z\le y$ such that $z$ covers $a_{i'}$ and $a_j$.
\end{enumerate}
\end{definition}

Let $\widehat{\WO}_n$ be the result of adding a new element to $\WO_n$, as a global maximum. Construct $\widehat{\WA}_n$ from $\WA_n$ in the same way. 
Let $\ZO_n$ and $\ZA_n$ be the dual posets of $\widehat{\WO}_n$ and $\widehat{\WA}_n$ respectively, i.e., replace every $\preceq$ with $\succeq$. By \cite[Theorem~5.12]{brady01}, $\ZO_n$ admits a recursive atom ordering. Our goal is to adapt that proof to the case of $\ZA_n$:

\begin{proposition}\label{zan:recursive}
 $\ZA_n$ admits a recursive atom ordering. 
\end{proposition}

First let us explain why Proposition~\ref{prop:hi_conn} for $\WA_n$ follows from Proposition~\ref{zan:recursive}.

\begin{proof}[Proof of Proposition~\ref{prop:hi_conn} for $\WA_n$]
Proposition~\ref{zan:recursive} and \cite[Lemma~5.3]{brady01} imply that 
$\ZA_n$ is homotopy Cohen--Macaulay. It is immediate from the definitions that the flag complex of $\WA_n$ is isomorphic to the link of the global minimum in $\ZA_n$, and that the property of being homotopy Cohen-Macaulay is inherited by links. Thus $\WA_n$ is homotopy Cohen--Macaulay as well.
\end{proof}

\begin{proof}[Proof of Proposition~\ref{zan:recursive}]
We first check that $\ZA_n$ is bounded and graded. Since $\ZA_n$ is the dual of $\widehat{\WA}_n$, this is equivalent to showing that $\widehat{\WA}_n$ is bounded and graded. Boundedness is clear: $\widehat{\WA}_n$ has a global maximum by construction and a global minimum since $\WA_n$ does. The fact that $\widehat{\WA}_n$ is graded follows directly from Observation~\ref{obs:degreebasic}(g).

Let us now prove that $\ZA_n$ has a recursive atom ordering. Since $\WA_n$ is defined as a subposet of $\WO_{n+1}$, we can view $\ZA_n$ as a subposet of $\ZO_{n+1}$, with the same global minimum, call it $\widehat{0}$. Clearly the subposet $\WA_n$ of $\WO_{n+1}$ is closed under taking lower bounds, since a folding cannot increase the degree of a labeled vertex, so the subposet $\ZA_n$ of $\ZO_{n+1}$ is closed under taking upper bounds. Hence, for any $T\preceq U$ in $\ZA_n$, the interval $[T,U]$ is the same viewed in either $\ZA_n$ or $\ZO_{n+1}$. 

Denote by $\mathbf{0}$ the unique tree in $\WO_{n+1}$ of degree $0$. Recall that it is the global minimum of $\WO_{n+1}$ and thus the global maximum of $\ZO_{n+1}$. By \cite[Lemmas~2.5 and~2.7]{brady01}, for every $\widehat{0}\ne T\in \ZO_{n+1}$, any ordering of the atoms of $[T,\mathbf{0}]$ in 
$\ZO_{n+1}$ is a recursive atom ordering, so by the previous paragraph the same holds in $\ZA_n$. In other words, the first condition in the definition of recursive atom ordering will always be automatically satisfied. 

As for the second condition, by Observation~\ref{obs:degreebasic}(f)(g), the posets $\ZA_n$ and $\ZO_{n+1}$ are both graded of the same length 
(namely $n-1$) and the atoms of $\ZA_n$ are precisely the atoms of $\ZO_{n+1}$ that lie in $\ZA_n$. Now we claim that the recursive atom ordering for $\ZO_{n+1}$ given in \cite{brady01} induces an ordering of the atoms of $\ZA_n$ that is also a recursive atom ordering. Say the atoms of $\ZA_n$ are ordered $T_1,\dots,T_\ell$. Let $i<j$ and say $T_i,T_j\preceq U$ for some $U\in\ZA_n$. Working in $\ZO_{n+1}$, where we already know we have a recursive atom ordering, we get that there exist $i'<j$ and $V\preceq U$ such that $V$ covers $T_{i'}$ and $T_j$. Since $\ZA_n$ is closed under upper bounds and $T_j\preceq V$ we know that $V\in\ZA_n$. The last thing to do is show that $T_{i'}\in\ZA_n$. 

The proof of \cite[Theorem~5.12]{brady01} essentially provides an algorithm for constructing $T_{i'}$ from $V$. In that proof the elements of $\ZO_{n+1}$ are viewed as trees rooted at $1$ and drawn upside down. According to this graphical representation foldings are naturally divided into two types, called lifts and merges, and the corresponding unfoldings are called drops and splits; see \cite[Definition~5.9]{brady01}, and also Figure~\ref{fig:drops_etc} for examples. The proof shows that $T_{i'}$ can be obtained from $V$ by a split (not a drop). 

On the other hand, if we view $V$ as an unrooted tree, then any unfolding at a leaf not labeled $1$ is a drop. Since $V\in \ZA_n$, the vertex $n+1$ in $V$ is a leaf not labeled $1$. Hence the unfolding that produces $T_{i'}$ from $V$ is not based at $n+1$, so $n+1$ remains a leaf in $T_{i'}$ and hence $T_{i'}\in\ZA_n$.
\end{proof}

\begin{figure}[htb]
\centering
\begin{tikzpicture}[line width=1pt]

\draw (0,0) -- (1,1) -- (2,0) -- (2,-1)   (1,1) -- (1,2)   (1,1) -- (1,0)   (1,-2) -- (2,-1) -- (3,-2);

\filldraw (1,2) circle (1.5pt);
\filldraw (0,0) circle (1.5pt);
\filldraw (1,1) circle (1.5pt);
\filldraw (1,0) circle (1.5pt);
\filldraw (2,-1) circle (1.5pt);
\filldraw (3,-2) circle (1.5pt);
\filldraw (2,0) circle (1.5pt);
\filldraw (1,-2) circle (1.5pt);

\node at (1,2.3) {$1$};
\node at (0,-0.3) {$2$};
\node at (1,-0.3) {$3$};
\node at (2.2,0.2) {$4$};
\node at (1,-2.3) {$5$};
\node at (3,-2.3) {$6$};

\node[scale=1.5] at (-2,2) {$\stackrel{\text{split at } 4}{\longleftarrow}$};
\node[scale=1.5] at (-2,0) {$\stackrel{\text{merge at } 4}{\longrightarrow}$};

\node[scale=1.5] at (4,0) {$\stackrel{\text{lift at } 2}{\longleftarrow}$};
\node[scale=1.5] at (4,2) {$\stackrel{\text{drop at } 2}{\longrightarrow}$};

\begin{scope}[xshift=-6cm]
\draw (0,0) -- (1,1) -- (3,-1) -- (3,-2)   (1,1) -- (1,2)   (1,1) -- (1,0)   (2,0) -- (1,-1) -- (1,-2);

\filldraw (1,2) circle (1.5pt);
\filldraw (0,0) circle (1.5pt);
\filldraw (1,1) circle (1.5pt);
\filldraw (1,0) circle (1.5pt);
\filldraw (3,-1) circle (1.5pt);
\filldraw (3,-2) circle (1.5pt);
\filldraw (2,0) circle (1.5pt);
\filldraw (1,-1) circle (1.5pt);
\filldraw (1,-2) circle (1.5pt);

\node at (1,2.3) {$1$};
\node at (0,-0.3) {$2$};
\node at (1,-0.3) {$3$};
\node at (2.2,0.2) {$4$};
\node at (1,-2.3) {$5$};
\node at (3,-2.3) {$6$};
\end{scope}

\begin{scope}[xshift=6cm]
\draw (0,0) -- (1,1) -- (2,0) -- (2,-1)   (1,1) -- (1,2)   (1,-2) -- (2,-1) -- (3,-2)   (0,0) -- (0,-2);

\filldraw (1,2) circle (1.5pt);
\filldraw (0,0) circle (1.5pt);
\filldraw (1,1) circle (1.5pt);
\filldraw (1,-2) circle (1.5pt);
\filldraw (3,-2) circle (1.5pt);
\filldraw (2,0) circle (1.5pt);
\filldraw (2,-1) circle (1.5pt);
\filldraw (0,-1) circle (1.5pt);
\filldraw (0,-2) circle (1.5pt);

\node at (1,2.3) {$1$};
\node at (0.2,-0.2) {$2$};
\node at (0,-2.3) {$3$};
\node at (2.2,0.2) {$4$};
\node at (1,-2.3) {$5$};
\node at (3,-2.3) {$6$};
\end{scope}

\end{tikzpicture}
\caption{An example of a split/merge, and a drop/lift.}
\label{fig:drops_etc}
\end{figure}

We are now ready to prove our main result.

\begin{proof}[Proof of parts (a) and (c) of Theorem~\ref{thm:mainPSA}]
The proofs of these two parts are analogous, so we will only prove (a), asserting that $\Sigma^{n-2}(\PSA_n)$ is dense in $\sphere(\PSA_n)$. Let $\chi$ be any character of $\PSA_n$ such that $\chi|_H\ne 0$ for all non-trivial $H\in\mathcal{HA}_n$. Let us verify the hypotheses of Theorem~\ref{thrm:MMV98}, to confirm that $[\chi]\in\Sigma^{n-2}(\PSA_n)$. We know that $\PSA_n$ and all the $H_A(T)$ are of type $\F_\infty$, hence $\F_{n-2}$, and that $\mathcal{HA}_n$ is non-empty, finite, intersection-closed (Proposition~\ref{prop:int_closed}), and $\infty$-generating, hence $(n-2)$-generating (Proposition~\ref{prop:contractible}). Since $\chi|_H\ne 0$ for all non-trivial $H\in \mathcal{HA}_n$, we have that $[\chi|_H]\in\Sigma^{n-2}(H)$ for all such $H$, since $H$ is abelian. Finally, $\calF_0(\mathcal{HA}_n)$ is $(n-3)$-connected by Proposition~\ref{prop:hi_conn}. All the hypotheses of Theorem~\ref{thrm:MMV98} have been confirmed, and so we conclude that $[\chi]\in\Sigma^{n-2}(\PSA_n)$. 

The last thing to show is that the set of all such $[\chi]$ is dense. It is clear from the McCool presentation that for any non-trivial $H_A(T)$, the image of $H_A(T)$ in the real vector space $\PSA_n/[\PSA_n,\PSA_n]\otimes \R$ has positive dimension. Thus, taking the dual, the subspace of $\Hom(\PSA_n,\R)$ consisting of characters that send $H_A(T)$ to $0$ has positive codimension. A finite union of subspaces with positive codimension has dense complement, so we are done.
\end{proof}

\section{More on $\Sigma^2$ for McCool groups}\label{sec:Sigma2}

Let us say that a character $\chi$ of $\PSA_n$ (respectively, $\PSO_n$) is \emph{generic} if $\chi|_H\ne 0$ for all non-trivial $H\in\mathcal{HA}_n$ (respectively, $H\in\mathcal{HO}_n$). It is not difficult to check that a character $\chi$ of $\PSA_n$ is generic if and only if $\chi(\alpha_{I,j})\neq 0$ for all (non-trivial) Whitehead generators $\alpha_{I,j}$ and likewise a character $\chi$ of $\PSO_n$ is generic if and only if $\chi(\overline{\alpha_{I,j}})\neq 0$ whenever $\overline{\alpha_{I,j}}\neq 1$ (recall that $\overline{\alpha_{I,j}}$ is the image of $\alpha_{I,j}$ in $\PSO_n$). The proof of Theorem~\ref{thm:mainPSA} shows that $\Sigma^{m}(\PSA_n)$ and $\Sigma^{m}(\PSO_n)$ contain (the equivalence classes of) all generic characters for $m\leq n-2$ and $m\leq n-3$, respectively, but does not say anything about non-generic characters.

The precise computation of $\Sigma^1(\PSA_n)$ in \cite{orlandikorner00} (see Theorem~\ref{thm:orlandikorner} below) suggests that $\Sigma^{m}(\PSA_n)$ should be much larger than the set of generic characters. In this section we will pose some specific questions about the structure of $\Sigma^{m}(\PSA_n)$ and then answer one of these questions for $m=2$. For brevity of exposition, we will limit our discussion to $\PSA_n$, but all the results and questions in this section have natural counterparts for $\PSO_n$.

\begin{theorem}[Orlandi-Korner]\label{thm:orlandikorner}
Let $n\geq 2$ and let $\chi$ be a non-zero character of $\PSA_n$. Then $[\chi]\not\in \Sigma^1(\PSA_n)$ if and only if one of the following holds:
\begin{itemize}
\item[(1)] There exist distinct indices $i,j$ such that $\chi(\alpha_{\ell,m})=0$ whenever $\{\ell,m\}\neq \{i,j\}$ as sets.
\item[(2)] There exist distinct indices $i,j,k$ such that $\chi(\alpha_{\ell,m})=0$ whenever $\{\ell,m\}\not\subseteq \{i,j,k\}$ and in addition $\chi(\alpha_{i,k})+\chi(\alpha_{j,k})=0$, $\chi(\alpha_{i,j})+\chi(\alpha_{k,j})=0$, and
$\chi(\alpha_{j,i})+\chi(\alpha_{k,i})=0$.
\end{itemize} 
\end{theorem}

Define the \emph{support} $\supp(\chi)$ of a character $\chi$ of $\PSA_n$ to be the set of all $(i,j)$ such that $\chi(\alpha_{i,j})\neq 0$. Thus, Theorem~\ref{thm:orlandikorner} implies in particular that the classes of characters whose support has at least 4 elements always lie in $\Sigma^1(\PSA_n)$, and one may ask if a similar phenomenon holds for other BNSR invariants.

\begin{question}\label{q:semigeneric}
Assume that $m\leq n-2$ and let $\chi$ a non-zero character of $\PSA_n$.
\begin{itemize}
\item[(a)] Does $\Sigma^m(\PSA_n)$ always contain $[\chi]$ whenever $|\supp(\chi)|$ is largest possible, that is, $\chi(\alpha_{i,j})\neq 0$ for all $i\neq j$?
\item[(b)] Does there exist a constant $C(m)$ and a subset $S(m,n)$ of $\{(i,j)\mid 1\leq i\neq j\leq n\}$ with $|S(m,n)|=C(m)$ such that $\Sigma^m(\PSA_n)$ always contains $[\chi]$ whenever $\supp(\chi)$ contains $S(m,n)$?
\item[(c)] Does there exist a constant $C(m)$ such that $\Sigma^m(\PSA_n)$ always contains $[\chi]$ whenever $|\supp(\chi)|>C(m)$?
\end{itemize}
\end{question}
Clearly, a positive answer to any part of Question~\ref{q:semigeneric} yields a positive answer to the preceding parts. We also note that
the conclusion of Theorem~\ref{thm:mainPSA} as well as each of the conditions (a)-(c) above yields a simple geometric condition on
the set $\sphere(\PSA_n)\setminus \Sigma^m(\PSA_n)$, the complement of $\Sigma^m(\PSA_n)$ in the character sphere of $\PSA_n$.

We first recall that $\sphere(\PSA_n)$ has dimension $n(n-1)-1$ and fix $m\leq n-2$. Theorem~\ref{thm:mainPSA} implies that $\sphere(\PSA_n)\setminus \Sigma^m(\PSA_n)$ is contained in the union of finitely many spheres of codimension $1$, but the number of spheres grows exponentially with $n$. A positive answer to (a) would reduce the number of spheres in the union to $n(n-1)$. Since $\Sigma^m(\PSA_n)$ is invariant under a natural action of the symmetric group $S_n$, it is easy to show that (b) would imply that $\sphere(\PSA_n)\setminus \Sigma^m(\PSA_n)$ is contained in the union of finitely many spheres of dimension $2C(m) n$. Finally, (c) would imply that the dimensions of the spheres can be bounded by a function of $m$ (independent of $n$).

\medskip

The main goal of this section is to answer Question~\ref{q:semigeneric}(b) in the affirmative for $m=2$ and $n\geq 10$. We will show that the set $S(2,n)=\{(1,2),(3,4),(5,6),(7,8),(9,10)\}$ has the required property; in other words, we will prove the following:

\begin{proposition} \label{prop:Sigma2sufficient}
Assume that $n\geq 10$. Then $[\chi]\in \Sigma^{2}(\PSA_n)$ whenever $\chi(\alpha_{i,j})\neq 0$ for all $(i,j)\in\{(1,2),(3,4),(5,6),(7,8), (9,10)\}$.
\end{proposition}

We will prove Proposition~\ref{prop:Sigma2sufficient} using the following general criterion: 

\begin{theorem}[Commutation $\Sigma^2$-criterion for quotients]\label{thm:Meinertcentral}
Let $Q=\langle X \mid R \rangle$ be a finitely presented group, let $R_1$ be a subset of $R$ and $G=\langle X \mid R_1\rangle$ (so that there is a natural surjection from $G$ to $Q$). Let $\chi$ be a character of $Q$ such that
\begin{itemize}
\item[(a)] $[\chi]\in\Sigma^2(G)$ (we denote the character of $G$ induced by $\chi$ by the same symbol) and
\item[(b)] for every $r\in R\setminus R_1$ there exists $g_r\in G$ with $\chi(g_r)\neq 0$ such that $g_r$ commutes with the image of $r$ in $G$.
\end{itemize}
Then $[\chi]\in\Sigma^2(Q)$.
\end{theorem}

Theorem~\ref{thm:Meinertcentral} is a straightforward consequence of a criterion of Meinert (see Theorem~\ref{thm:Meinert} in Appendix~A) established in his Ph.D. thesis \cite{Meinertthesis}; see also a remark after the proof of Corollary~2.8 in \cite{Meinert97} where Theorem~\ref{thm:Meinert} is stated in a different language. Since \cite{Meinertthesis} is not easily accessible, we will provide a proof of Theorem~\ref{thm:Meinert} in Appendix~A where we will also explain why it implies 
Theorem~\ref{thm:Meinertcentral}.

\begin{remark}
Theorem~\ref{thm:Meinertcentral} is particularly convenient to apply to a group $Q=\langle X\mid R \rangle$ in the case when the majority of relations in $R$ are of the form $[x,y]=1$ for $x,y\in X$; call such relations \emph{RAAG-like}. In this case we can let $R_1$ be a set of some or all of the RAAG-like relations from $R$, so that $G=\langle X \mid R_1\rangle$ is a RAAG, and therefore $\Sigma^2(G)$ is completely determined by the main result of \cite{meier98} (see Theorem~\ref{thm:RAAGcriterion} below). This allows us to get a handle on condition (a). As for condition (b), a natural way to ensure this holds is to require that for each relator $r\in R\setminus R_1$ there exists a generator $x_r\in X$ such that, for every generator $x$ appearing in $r$, the (RAAG-like) relation $[x_r,x]=1$ holds and lies in $R_1$. Indeed, in this case if we set $g_r=x_r$, then condition (b) holds for any $\chi$ satisfying $\chi(x_r)\ne 0$. The more RAAG-like relations from $R$ are included in $R_1$, the easier it is to satisfy (b); however (unlike for $\Sigma^1$), including too many of the RAAG-like relations in $R_1$ can affect whether (a) holds. Thus, in practice it can be a balancing act to find the most useful $R_1$.
\end{remark}

Theorem~\ref{thm:RAAGcriterion} below describes the $m=1$ and $m=2$ cases of the full computation of the BNSR-invariant $\Sigma^m$ for an arbitrary RAAG, proved by Meier, Meinert, and Van Wyk in \cite{meier98}, and also see \cite{bux99} for an alternative proof due to Bux and Gonzalez. Before stating Theorem~\ref{thm:RAAGcriterion}, we introduce some standard notation and terminology. Let $\Gamma$ be a finite graph with no loops or multiple edges, and let $\widehat{\Gamma}$ be the flag complex associated to $\Gamma$ (that is, the simplices of $\widehat{\Gamma}$ are precisely the cliques in $\Gamma$). Let $A_\Gamma$ be the RAAG corresponding to $\Gamma$, that is, $A_{\Gamma}$ has generators $\{x_v\mid  v \mbox{ is a vertex of }\Gamma\}$ and defining relations $[x_v,x_w]=1$ whenever $v$ and $w$ are adjacent.
\vskip .12cm

Let $\chi$ be a non-zero character of $A_{\Gamma}$.
\begin{itemize}
\item A vertex $v$ of $\Gamma$ is called \emph{living} (with respect to $\chi$) if $\chi(v)\neq 0$ and \emph{dead} otherwise.
\item The \emph{living subcomplex of $\chi$} is the full subcomplex of $\widehat{\Gamma}$ spanned by the living vertices.
\item A simplex $\sigma$ of $\widehat{\Gamma}$ is called \emph{dead} if all its vertices are dead.
\item The \emph{living link} of a simplex $\sigma$ in $\widehat{\Gamma}$ is the full subcomplex of the link of $\sigma$ spanned by its living vertices.
\end{itemize}

\begin{theorem}\label{thm:RAAGcriterion}\cite{meier98,bux99}
Let $\Gamma$ and $A_{\Gamma}$ be as above and $\chi$ a non-zero character of $A_{\Gamma}$.
The following hold:
\begin{enumerate}
\item[(1)] $[\chi]\in \Sigma^1(A_{\Gamma})$ if and only if the living subcomplex of $\chi$ is connected and every dead vertex is adjacent to a living vertex;
\item[(2)] $[\chi]\in \Sigma^2(A_{\Gamma})$ if and only if
\begin{itemize}
\item the living subcomplex of $\chi$ is simply-connected,
\item the living link of every dead vertex is connected and non-empty, and
\item the living link of every dead edge is non-empty.
\end{itemize}
\end{enumerate}
\end{theorem}

Condition (2) is often difficult to check in specific examples; however, its verification can sometimes be avoided entirely as there is a natural condition on $\Gamma$ which guarantees that $\Sigma^1(A_{\Gamma})=\Sigma^2(A_{\Gamma})$, namely the following:

\begin{definition}\label{chordal}
A graph $\Gamma$ is called \emph{chordal} if any cycle $C$ of length $\geq 4$ in $\Gamma$ has a chord, that is, an edge of $\Gamma$ connecting non-consecutive vertices of $C$.
\end{definition}

Whenever $\Gamma$ is chordal we have $\Sigma^1(A_{\Gamma})=\Sigma^2(A_{\Gamma})$ \cite[Corollary~6.1]{meier98}, which makes understanding $\Sigma^2(A_{\Gamma})$ much easier in this case. We now use RAAGs corresponding to chordal graphs to prove our main result of this section, Proposition~\ref{prop:Sigma2sufficient}.

\begin{proof}[Proof of Proposition~\ref{prop:Sigma2sufficient}]
Let $S=\{(1,2),(3,4),(5,6),(7,8),(9,10)\}$ and define the graph $\Gamma$ as follows. The vertices of $\Gamma$ are all the pairs $(i,j)$ with $1\leq i\neq j\leq n$. Two distinct vertices $(i,j)$ and $(k,\ell)$ are adjacent if and only if at least one of them lies in $S$ and $[\alpha_{i,j},\alpha_{k,l}]=1$ in $\PSA_n$. Note that
\begin{itemize}
\item[(i)] any two vertices in $S$ are adjacent to each other,
\item[(ii)] every vertex of $\Gamma$ is adjacent to a vertex in $S$, and
\item[(iii)] any vertex of $\Gamma$ not in $S$ can only be adjacent to vertices that are in $S$.
\end{itemize}
If $C$ is a cycle in $\Gamma$, then either all vertices of $C$ are in $S$ or there is a vertex of $C$ outside of $S$, in which case both of its neighbors must be in $S$ by property (iii); either way, by (i) $C$ contains a chord as soon as its length is at least 4. Thus, $\Gamma$ is chordal.

Now let $Q=\PSA_n$ (with the McCool presentation) and $G=A_{\Gamma}$. By construction, $Q$ is a quotient of $G$. Let $\chi$ be a character of $Q$ such that $\chi(\alpha_{i,j})\neq 0$ whenever $(i,j)\in S$. We need to check that the induced character of $G$ (also denoted by $\chi$) satisfies the hypotheses of Theorem~\ref{thm:Meinertcentral}. The living subcomplex of $\chi$ contains $S$ and hence is connected by properties (i) and (ii) above. Likewise, property (ii) implies that every dead vertex is adjacent to a living vertex. Hence by Theorem~\ref{thm:RAAGcriterion}(1), $[\chi]$ lies in $\Sigma^1(G)$, and $\Sigma^1(G)=\Sigma^2(G)$ since $\Gamma$ is chordal. Thus we have confirmed condition (a) of Theorem~\ref{thm:Meinertcentral}. We now check (b); actually, we will show that it holds for all $r\in R$, where $R$ is the set of defining relators in the McCool presentation. This is because any $r\in R$ involves at most four indices, but $S$ has five elements, so there exists $(k,\ell)\in S$ such that neither $k$ nor $\ell$ appears in the subscript of any generator used in $r$, whence the corresponding generator $x_{(k,\ell)}$ of $G$ commutes (in $G$) with every generator used in~$r$.
\end{proof}

\appendix
\section{A proof of Meinert's $\Sigma^2$-criterion}\label{app:meinert}

The goal of this appendix is to state and prove a criterion of Meinert established in his Ph.D. thesis \cite[Satz~5.2.1]{Meinertthesis} and that is the key to proving Theorem~\ref{thm:Meinertcentral}.

In order to state Meinert's criterion, we need the notion of a generating set for groups with operators. Let $H$ be a group, and suppose we are given a set $\Phi$ together with a map $\Phi\to\Aut(H)$. We will say that $H$ is \emph{finitely generated as a $\Phi$-group} if there exists a finite subset $S$ of $H$ such that the smallest $\Phi$-invariant subgroup of $H$ containing $S$ is $H$ itself. If $H$ is a normal subgroup of some group $G$ and $\Phi$ is a subset of $G$, we can view $H$ as a $\Phi$-group where $\Phi$ acts on $H$ by conjugation.

\begin{theorem}[Meinert's $\Sigma^2$-criterion for quotients]\label{thm:Meinert}
Let $1\to N\to G\to Q\to 1$ be a short exact sequence of groups, with $G$ and $Q$ finitely presented. Let $\chi$ be a non-zero character of $Q$. Assume that
\begin{itemize}
\item[(i)] $[\chi]\in\Sigma^2(G)$ (as before we denote the character of $G$ induced by $\chi$ by the same symbol)
\item[(ii)] $N$ is finitely generated as a $G_{\chi\geq 0}$-group, where $G_{\chi\geq 0}=\{g\in G\mid \chi(g)\geq 0\}$.
\end{itemize}
Then $[\chi]\in \Sigma^2(Q)$.
\end{theorem}

This criterion should be compared to the classical fact that $[\chi]\in\Sigma^1(G)$ always implies $[\chi]\in\Sigma^1(Q)$ \cite[Proposition~3.3]{bieri87}.

Let us first deduce Theorem~\ref{thm:Meinertcentral} from Theorem~\ref{thm:Meinert}.

\begin{proof}[Proof of Theorem~\ref{thm:Meinertcentral}] Assume that $G$ and $Q$ are as in Theorem~\ref{thm:Meinertcentral}. Then we have the short exact sequence $1\to N\to G\to Q\to 1$ where $N$ is the normal closure of the image of $R\setminus R_1$ in $G$. Condition (i) in Theorem~\ref{thm:Meinert} holds by assumption, so we just need to check (ii). By definition any element $n\in N$ can be written as $n=\prod_{i=1}^k a_i r_i^{\pm 1} a_i^{-1}$ where $r_i\in R\setminus R_1$ and $a_i\in G$ (here we are equivocating between $r_i$ and its image in $G$). Recall that by the hypotheses of Theorem~\ref{thm:Meinertcentral}, for each $i$ there exists $g_i$ such that $\chi(g_i)\neq 0$ and $[g_i,r_i]=1$ in $G$. The latter condition means that in the above expression for $n$ we can replace $a_i$ by $a_i g_i^{m_i}$ for any integer $m_i$, and since $\chi(g_i)\neq 0$, after doing so we can assume that $a_i\in G_{\chi\geq 0}$. But this means that $N$ is generated as a $G_{\chi\geq 0}$-group by the finite set $R\setminus R_1$, so hypothesis (ii) in Theorem~\ref{thm:Meinert} holds, as desired.
\end{proof}

We now make some preparations for the proof of Theorem~\ref{thm:Meinert}. We will start with some terminology and notation. Let $X$ be a set. As usual, by $F(X)$ we will denote the free group with basis $X$. By $\Omega(X)$ we will denote the set of all finite words in the alphabet $X\sqcup X^{-1}$ (not necessarily reduced). We can think of $F(X)$ as both a subset and a quotient of $\Omega(X)$ (the quotient map $\Omega(X)\to F(X)$ is a homomorphism of monoids). While no non-empty word in $\Omega(X)$ is invertible in $\Omega(X)$, there is a natural involution $w\mapsto w^{-1}$ on $\Omega(X)$: if $w=x_1\cdots x_n\in \Omega(X)$, with $x_i\in X^{\pm1}$ for all $i$, we set 
$w^{-1}=x_n^{-1}\cdots x_1^{-1}$ where by convention $(x^{-1})^{-1}=x$ for all $x\in X$.

Let $R$ be a subset of $\Omega(X)$. By $R^{F(X)}$ we will denote the normal closure of the image of $R$ in $F(X)$. By $R^{\Omega(X)}$ we will denote the set of all words of the form $\prod_{i=1}^k a_i r_i a_i^{-1}$ where $a_i\in \Omega(X)$ and $r_i\in R^{\pm 1}$. Thus, $R^{\Omega(X)}$ maps onto $R^{F(X)}$ under the natural projection $\Omega(X)\to F(X)$.

\begin{definition}
Given a word $w=x_1\cdots x_n\in\Omega(X)$ with $x_i\in X^{\pm 1}$ for all $i$, by a \emph{prefix} of $w$ we will mean a subword of the form $x_1\cdots x_k$ with $k\leq n$. We do allow $k=0$.
\end{definition}

If $X$ is a generating set of some group $G$, any character of $G$ naturally induces a character of the group $F(X)$ and a character of the monoid $\Omega(X)$ (for simplicity we will not introduce separate notation for those induced characters).
 
Given $w\in \Omega(X)$ and a character $\chi$ of $\Omega(X)$, define
\[
\chi_{min}(w)= \min\{\chi(v)\mid v\text{ is a prefix of }w\}\text{.}
\]
We will say that $w$ is \emph{$\chi$-non-negative} if $\chi_{min}(w)= 0$ (since we count the empty word as a prefix, $\chi_{min}(w)$ is never positive).

The following observation collects some easy properties of the $\chi_{min}$ function.

\begin{observation}\label{obs:easy}
Let $X$ be a set, and let $u,v\in \Omega(X)$. The following hold:
\begin{itemize}
    \item[(a)] $\chi_{min}(uv)= \min\{\chi_{min}(u), \chi(u)+\chi_{min}(v)\}\geq \chi_{min}(u)+\chi_{min}(v)$.
    \item[(b)] (special case of (a)). If $\chi(u)=0$, then $\chi_{min}(uv)= \min\{\chi_{min}(u),\chi_{min}(v)\}$.
    \item[(c)] If $\chi(u)=0$, then $\chi_{min}(u^{-1})=\chi_{min}(u)$.
\end{itemize}
\end{observation}

We can rephrase the usual, topological definition of $\Sigma^2(G)$ combinatorially in terms of $\chi_{min}$ as follows.

\begin{lemma}\label{lem:combinatorial}
Let $\langle X\mid R\rangle$ be a finite presentation of a group $G$, and let $\chi$ be a non-zero character of $G$. Then $[\chi]\in \Sigma^2(G)$ if and only if $[\chi]\in \Sigma^1(G)$ and 
\begin{itemize}
\item[\Scombp]{ }{\rm :} there is a constant $C$ such that if $w\in \Omega(X)$ is any $\chi$-non-negative relator of $G$, then there exists $w'\in R^{\Omega(X)}$ such that $w'=w$ in $F(X)$ and $\chi_{min}(w')\geq C$.
\end{itemize}
(Note that we must have $C\leq 0$ since $\chi_{min}(w')$ is never positive.)
\end{lemma}

\begin{proof}
Let $K$ be the Cayley complex of $G$ corresponding to the finite presentation $\langle X\mid R\rangle$. Let $h_\chi\colon K\to\R$ be as in the definition of $\Sigma$-invariants. 

First recall that $\Sigma^2(G)$ is contained in $\Sigma^1(G)$, and for a character $\chi$ we have $[\chi]\in \Sigma^1(G)$ if and only if $K^{h_\chi\ge t}$ is connected for some (or, equivalently, all) $t\in\mathbb R$. Thus, Lemma~\ref{lem:combinatorial} only deals with characters $\chi$ such that $[\chi]\in \Sigma^1(G)$, and in view of Remark~\ref{rem:BNSR-invariants}(b), it suffices to show that the following are equivalent for a given constant $C\leq 0$:
\begin{itemize}
\item[(a)] the inclusion $K^{h_\chi\ge 0}\to K^{h_\chi\ge C}$ induces the trivial map 
$\pi_1(K^{h_\chi\ge 0},1_G)\to \pi_1(K^{h_\chi\ge C},1_G)$ (note that $1_G$ lies in both $K^{h_\chi\ge 0}$ and $K^{h_\chi\ge 0}$).
\item[(b)](\Scomb) holds for $C$.
\end{itemize}

We will use the following standard fact about the fundamental groups of CW-complexes. Let $Y$ be a path-connected CW-complex and fix a vertex $v_0$ of $Y$. Then every element of $\pi_1(Y,v_0)$ is represented by an edge path, and if $p$ is any closed edge path starting from $v_0$, then $p$ represents the identity in $\pi_1(Y,v_0)$ if and only if there exists another closed edge path $p'$ starting from $v_0$ such that
\begin{itemize}
\item[(1)] $p'$ and $p$ represent the same element in $\pi_1(Y^1,v_0)$ where $Y^1$ is the $1$-skeleton of $Y$;
\item[(2)] $p'$ can be written as a product $\prod\limits_{i=1}^k\alpha_i \beta_i \alpha_i^{-1}$ where for each $i$ the path $\alpha_i$ is an edge path from $v_0$ to some vertex $v_i$ of $Y$ and there exists a $2$-cell $R_i$ containing $v_i$ such that $\beta_i$ is a closed path which starts at $v_i$ and traverses the boundary of $R_i$.
\end{itemize}

We are now ready to prove the equivalence of conditions (a) and (b) above. Below we will establish the implication  ``(a)$\Rightarrow$ (b)''. The other implication can be proved by just reversing the argument.

Suppose (a) holds. Let $w\in \Omega(X)$ be a $\chi$-non-negative relator of $G$, and let $p$ be the unique edge path in $\Cay(G,X)$ starting at $1_G$ whose label is $w$. Note that $p$ is closed since $w$ is a relator of $G$.

The assumption that $w$ is $\chi$-non-negative means precisely that $p$ lies in $K^{h_\chi\ge 0}$, so by our choice of $C$, the path $p$ represents the identity in 
$\pi_1(K^{h_\chi\ge C},1_G)$. Thus, there exists a closed edge path $p'$ in $Y=K^{h_\chi\ge C}$ satisfying (1) and (2) above. Let $w'\in \Omega(X)$ denote the label of $p'$. Condition (1) means precisely that $w=w'$ in $F(X)$, and condition (2) means precisely that $w'\in R^{\Omega(X)}$ with $\chi_{min}(w')\geq C$. Thus, $w'$ satisfies the conclusion of (\Scomb), and therefore (b) holds.
\end{proof}

\medskip

We can now prove Meinert's criterion.

\begin{proof}[Proof of Theorem~\ref{thm:Meinert}]
To reiterate the setup, we have a short exact sequence of groups $1\to N\to G\to Q\to 1$ with $G$ and $Q$ finitely presented, and a non-zero character $\chi\colon Q\to\R$, with $\chi$ also denoting the induced character of $G$. We are assuming that $[\chi]\in\Sigma^2(G)$ and that $N$ is finitely generated as a $G_{\chi\ge 0}$-group. Our goal is to prove that $[\chi]\in\Sigma^2(Q)$.

Since $[\chi]\in\Sigma^1(G)$, we clearly have $[\chi]\in \Sigma^1(Q)$ \cite[Proposition~3.3]{bieri87}. Hence by Lemma~\ref{lem:combinatorial} it suffices to show that $\chi$ considered as a character of $Q$ satisfies (\Scomb).

Fix a finite presentation $\langle X\mid R\rangle$ of $G$. Choose a finite subset $R_Q$ of $\Omega(X)$ whose image in $G$ generates $N$ as a $G_{\chi\geq 0}$-group, and set
\[
C_Q = \min\{\chi_{min}(r)\mid r\in R_Q\}\quad\text{ and}\quad
C_X = \max\{\chi(x)\mid x\in X\sqcup X^{-1}\}\text{.}
\]
Note that $C_Q\le 0$ and $C_X>0$. Next, let $C_G\leq 0$ be a constant for $\chi$ considered as a character of $G$ in (\Scomb), and set
\[
C = C_Q - C_X + C_G\text{.}
\]
Note that $C<C_Q$.
\vskip .1cm

We will show that (\Scomb) holds for $\chi$ as a character of $Q$ using this constant $C$ and the finite presentation $Q=\langle X\mid R\cup R_Q\rangle$. Let $w\in\Omega(X)$ be a $\chi$-non-negative relator of $Q$. The image of $w$ in $G$ lies in $N$, so we can choose $v\in\Omega(X)$ with $v=w$ in $G$ such that $v=\prod_{i=1}^k b_i r_i b_i^{-1}$, for some $r_i\in R_Q^{\pm 1}$ and $b_i$ satisfying $\chi(b_i)\geq 0$ for each $i$. Since $[\chi]\in \Sigma^1(G)$, without loss of generality $\chi_{min}(b_i)\geq 0$ for all $i$; by Observation~\ref{obs:easy}(b) this ensures that $\chi_{\min}(v)\geq C_Q$. Now set $u= wv^{-1}$, so $u$ is a relator of $G$, and since $\chi_{min}(w)=0$ by assumption, we compute using Observation~\ref{obs:easy}(b)(c) that
\[
\chi_{min}(u)\geq \min\{\chi_{min}(w),\chi_{min}(v^{-1})\}=\chi_{min}(v^{-1})=\chi_{min}(v)\geq C_Q\text{.}
\]
Next choose $x\in X\sqcup X^{-1}$ with $\chi(x)>0$, and choose $n\ge 0$ minimal such that $\chi(x^n)\geq -C_Q$. By minimality, $\chi(x^n)< -C_Q + C_X$. Set $u'= x^n u x^{-n}$, so $u'$ is a $\chi$-non-negative relator of $G$. Since $[\chi]\in\Sigma^2(G)$ we can choose $u''\in R^{\Omega(X)}$ such that $u'=u''$ in $F(X)$ and $\chi_{min}(u'')\geq C_G$. Say $u''=\prod_{j=1}^l a_j r_j a_j^{-1}$ for some $a_j\in\Omega(X)$ and $r_j\in R^{\pm 1}$. Since $\chi(a_j r_j a_j^{-1})=0$ for all $j$, by Observation~\ref{obs:easy}(b) we get
\[
\min_j \{\chi_{min}(a_j r_j a_j^{-1})\}=\chi_{min}(u'')\geq C_G \text{.}
\]

But now we can rewrite $w$ as an element of $F(X)$ as follows:
\[
w=uv=(x^{-n}u' x^n)v=(x^{-n}u'' x^n)v = \left(\prod_{j=1}^l x^{-n} a_j r_j a_j^{-1}x^{n}\right)\cdot v \text{.}
\]
Let $w'\in\Omega(X)$ denote this last word. Since $v\in R_Q^{\Omega(X)}$, we have that $w'$ lies in $(R\cup R_Q)^{\Omega(X)}$ and $w=w'$ in $F(X)$. The only thing remaining is to prove that $\chi_{min}(w')\ge C$. Using Observation~\ref{obs:easy}(b) again, we have
\[
\chi_{min}(w') = \min\{\chi_{min}(v), \chi_{min}(x^{-n} a_j r_j a_j^{-1}x^{n})\mid 1\leq j\leq l\}\text{.}
\]
If $\chi_{min}(w')=\chi_{min}(v)$, then $\chi_{min}(w')\ge C_Q> C$ and we are done. Now suppose that $\chi_{min}(w')=\chi_{min}(x^{-n} a_j r_j a_j^{-1}x^{n})$. By Observation~\ref{obs:easy}(a) we have
\[
\chi_{min}(x^{-n} a_j r_j a_j^{-1}x^{n})\geq \chi_{min}(x^{-n})+\chi_{min}(a_j r_j a_j^{-1})+\chi_{min}(x^{n})\text{.}
\]
By construction we have $\chi_{min}(x^{-n})> C_Q - C_X$, $\chi_{min}(x^{n})= 0$, and $\chi_{min}(a_j r_j a_j^{-1})\geq C_G$, so putting everything together we get $\chi_{min}(w')\geq C_Q-C_X+C_G = C$, as desired.
\end{proof}

\bibliographystyle{alpha}

\begin{thebibliography}{BGCAMP22}

\bibitem[AGMP25]{ardaizgale}
Peio Ardaiz-Galé and Conchita Martínez-Pérez.
\newblock The McCullough--Miller complex for right angled Artin groups.
\newblock arXiv:2506.03377.

\bibitem[AH93]{abels93}
Herbert Abels and Stephan Holz.
\newblock Higher generation by subgroups.
\newblock {\em J. Algebra}, 160(2):310--341, 1993.

\bibitem[BFM{\etalchar{+}}16]{bux16}
Kai-Uwe Bux, Martin~G. Fluch, Marco Marschler, Stefan Witzel, and Matthew C.~B. Zaremsky.
\newblock The braided {T}hompson's groups are of type {$\rm F_\infty$}.
\newblock {\em J. Reine Angew. Math.}, 718:59--101, 2016.
\newblock With an appendix by Zaremsky.

\bibitem[BG99]{bux99}
Kai-Uwe Bux and Carlos Gonzalez.
\newblock The {B}estvina-{B}rady construction revisited: geometric computation of {$\Sigma$}-invariants for right-angled {A}rtin groups.
\newblock {\em J. London Math. Soc. (2)}, 60(3):793--801, 1999.

\bibitem[BG03]{bieri03}
Robert Bieri and Ross Geoghegan.
\newblock Connectivity properties of group actions on non-positively curved spaces.
\newblock {\em Mem. Amer. Math. Soc.}, 161(765):xiv+83, 2003.

\bibitem[BGCAMP22]{blascogarcia22}
Rub\'en Blasco-Garc\'ia, Jos\'e{}~Ignacio Cogolludo-Agust\'in, and Conchita Mart\'inez-P\'erez.
\newblock On the {S}igma-invariants of even {A}rtin groups of {FC}-type.
\newblock {\em J. Pure Appl. Algebra}, 226(7):Paper No. 106984, 20, 2022.

\bibitem[BGK10]{bieri10}
Robert Bieri, Ross Geoghegan, and Dessislava~H. Kochloukova.
\newblock The sigma invariants of {T}hompson's group {$F$}.
\newblock {\em Groups Geom. Dyn.}, 4(2):263--273, 2010.

\bibitem[BH13]{brendle13}
Tara~E. Brendle and Allen Hatcher.
\newblock Configuration spaces of rings and wickets.
\newblock {\em Comment. Math. Helv.}, 88(1):131--162, 2013.

\bibitem[BMMM01]{brady01}
Noel Brady, Jon McCammond, John Meier, and Andy Miller.
\newblock The pure symmetric automorphisms of a free group form a duality group.
\newblock {\em J. Algebra}, 246(2):881--896, 2001.

\bibitem[BNS87]{bieri87}
Robert Bieri, Walter~D. Neumann, and Ralph Strebel.
\newblock A geometric invariant of discrete groups.
\newblock {\em Invent. Math.}, 90(3):451--477, 1987.

\bibitem[BR88]{bieri88}
Robert Bieri and Burkhardt Renz.
\newblock Valuations on free resolutions and higher geometric invariants of groups.
\newblock {\em Comment. Math. Helv.}, 63(3):464--497, 1988.

\bibitem[Bro82]{brown82}
Kenneth~S. Brown.
\newblock {\em Cohomology of groups}, volume~87 of {\em Graduate Texts in Mathematics}.
\newblock Springer-Verlag, New York-Berlin, 1982.

\bibitem[BT68]{baumslag68}
Gilbert Baumslag and Tekla Taylor.
\newblock The centre of groups with one defining relator.
\newblock {\em Math. Ann.}, 175:315--319, 1968.

\bibitem[Bux04]{bux04}
Kai-Uwe Bux.
\newblock Finiteness properties of soluble arithmetic groups over global function fields.
\newblock {\em Geom. Topol.}, 8:611--644, 2004.

\bibitem[BW83]{bjoerner83}
Anders Bj\"orner and Michelle Wachs.
\newblock On lexicographically shellable posets.
\newblock {\em Trans. Amer. Math. Soc.}, 277(1):323--341, 1983.

\bibitem[BWC07]{baez07}
John~C. Baez, Derek~K. Wise, and Alissa~S. Crans.
\newblock Exotic statistics for strings in 4{D} {$BF$} theory.
\newblock {\em Adv. Theor. Math. Phys.}, 11(5):707--749, 2007.

\bibitem[CGJ05]{chen05}
Yuqing Chen, Henry~H. Glover, and Craig~A. Jensen.
\newblock Proper actions of automorphism groups of free products of finite groups.
\newblock {\em Internat. J. Algebra Comput.}, 15(2):255--272, 2005.

\bibitem[Col89]{collins89}
Donald~J. Collins.
\newblock Cohomological dimension and symmetric automorphisms of a free group.
\newblock {\em Comment. Math. Helv.}, 64(1):44--61, 1989.

\bibitem[Dam17]{damiani17}
Celeste Damiani.
\newblock A journey through loop braid groups.
\newblock {\em Expo. Math.}, 35(3):252--285, 2017.

\bibitem[EFMP24]{escartin-ferrer24}
Marcos Escart\'in-Ferrer and Conchita Mart\'inez-Perez.
\newblock On the {$\Sigma$}-invariants of {A}rtin groups satisfying the {$K(\pi, 1)$}-conjecture.
\newblock {\em J. Lond. Math. Soc. (2)}, 109(1):Paper No. e12861, 35, 2024.

\bibitem[FM12]{farbmargalit}
Benson Farb and Dan Margalit.
\newblock {\em A primer on mapping class groups}, volume~49 of {\em Princeton Mathematical Series}.
\newblock Princeton University Press, Princeton, NJ, 2012.

\bibitem[Gri13]{griffin13}
James~T. Griffin.
\newblock Diagonal complexes and the integral homology of the automorphism group of a free product.
\newblock {\em Proc. Lond. Math. Soc. (3)}, 106(5):1087--1120, 2013.

\bibitem[HK24]{hughes_kielak}
Sam Hughes and Dawid Kielak.
\newblock {BNSR} invariants and $\ell^2$-homology.
\newblock arXiv:2401.05545.

\bibitem[JMM06]{jensen06}
Craig Jensen, Jon McCammond, and John Meier.
\newblock The integral cohomology of the group of loops.
\newblock {\em Geom. Topol.}, 10:759--784, 2006.

\bibitem[KM22]{kochloukova22}
Dessislava~Hristova Kochloukova and Luis Mendon{\c c}a.
\newblock On the {B}ieri-{N}eumann-{S}trebel-{R}enz {$\Sigma$}-invariants of the {B}estvina-{B}rady groups.
\newblock {\em Forum Math.}, 34(3):605--626, 2022.

\bibitem[KMM15]{koban15}
Nic Koban, Jon McCammond, and John Meier.
\newblock The {BNS}-invariant for the pure braid groups.
\newblock {\em Groups Geom. Dyn.}, 9(3):665--682, 2015.

\bibitem[LI24]{llosa_isenrich}
Claudio Llosa~Isenrich.
\newblock From the second {BNSR} invariant to {D}ehn functions of coabelian subgroups.
\newblock arXiv:2404.12334.

\bibitem[L{\"u}c02]{lueck02}
Wolfgang L{\"u}ck.
\newblock {\em {$L^2$}-invariants: theory and applications to geometry and {$K$}-theory}, volume~44 of {\em Ergebnisse der Mathematik und ihrer Grenzgebiete. 3. Folge. A Series of Modern Surveys in Mathematics [Results in Mathematics and Related Areas. 3rd Series. A Series of Modern Surveys in Mathematics]}.
\newblock Springer-Verlag, Berlin, 2002.

\bibitem[McC86]{mccool86}
J.~McCool.
\newblock On basis-conjugating automorphisms of free groups.
\newblock {\em Canad. J. Math.}, 38(6):1525--1529, 1986.

\bibitem[Mei93]{Meinertthesis}
Holger Meinert.
\newblock {\em Die h\"oheren geometrischen Invarianten $\Sigma^m$ von Gruppen via Operationen auf CW-Komplexen und der Beweis der $\Sigma^m$-Vermutung f\"ur metabelsche Gruppen endlichen Pr\"ufferr-Ranges}.
\newblock PhD thesis, Universit\"at Frankfurt a.M.,, 1993.

\bibitem[Mei97]{Meinert97}
Holger Meinert.
\newblock Actions on {$2$}-complexes and the homotopical invariant {$\Sigma^2$} of a group.
\newblock {\em J. Pure Appl. Algebra}, 119(3):297--317, 1997.

\bibitem[MM96]{mccullough96}
Darryl McCullough and Andy Miller.
\newblock Symmetric automorphisms of free products.
\newblock {\em Mem. Amer. Math. Soc.}, 122(582):viii+97, 1996.

\bibitem[MM04]{mccammond04}
Jon McCammond and John Meier.
\newblock The hypertree poset and the $l^2$-Betti numbers of the motion group of the trivial link.
\newblock {\em Math. Ann.}, 328(4):633--652, 2004.

\bibitem[MMV98]{meier98}
John Meier, Holger Meinert, and Leonard VanWyk.
\newblock Higher generation subgroup sets and the {$\Sigma$}-invariants of graph groups.
\newblock {\em Comment. Math. Helv.}, 73(1):22--44, 1998.

\bibitem[MMV01]{meier01}
John Meier, Holger Meinert, and Leonard VanWyk.
\newblock On the {$\Sigma$}-invariants of {A}rtin groups.
\newblock volume 110, pages 71--81. 2001.
\newblock Geometric topology and geometric group theory (Milwaukee, WI, 1997).

\bibitem[OK00]{orlandikorner00}
Lisa~A. Orlandi-Korner.
\newblock The {B}ieri-{N}eumann-{S}trebel invariant for basis-conjugating automorphisms of free groups.
\newblock {\em Proc. Amer. Math. Soc.}, 128(5):1257--1262, 2000.

\bibitem[Ren88]{renz_thesis}
Burkhardt Renz.
\newblock {\em Geometrische Invarianten und Endlichkeitseigenschaften von Gruppen}.
\newblock PhD thesis, Universität Frankfurt, 1988.

\bibitem[Str12]{strebel_book}
Ralph Strebel.
\newblock Notes on the {S}igma invariants.
\newblock arXiv:1204.0214.

\bibitem[Zar17]{zaremsky17PB_n}
Matthew C.~B. Zaremsky.
\newblock Separation in the {BNSR}-invariants of the pure braid groups.
\newblock {\em Publ. Mat.}, 61(2):337--362, 2017.

\bibitem[Zar18]{zaremsky18}
Matthew C.~B. Zaremsky.
\newblock Symmetric automorphisms of free groups, {BNSR}-invariants, and finiteness properties.
\newblock {\em Michigan Math. J.}, 67(1):133--158, 2018.

\end{thebibliography}

\newcommand{\etalchar}[1]{$^{#1}$}


\end{document}